\documentclass[a4paper,11pt]{article}
\usepackage{amsmath, amsfonts, amscd, amssymb, amsthm, enumerate, xypic}

\def\bfB{\mathbf{B}}

\def\Ortho{\mathrm{O}}

\newcommand{\Hom}{\operatorname{Hom}}
\newcommand{\Mat}{\operatorname{M}}

\newcommand{\id}{\operatorname{id}}
\newcommand{\GL}{\operatorname{GL}}
\newcommand{\Rad}{\operatorname{Rad}}

\newcommand{\Irr}{\operatorname{Irr}}
\newcommand{\Nil}{\operatorname{Nil}}
\newcommand{\Co}{\operatorname{Co}}
\newcommand{\SL}{\operatorname{SL}}
\newcommand{\Ker}{\operatorname{Ker}}
\newcommand{\End}{\operatorname{End}}

\newcommand{\Vect}{\operatorname{span}}
\newcommand{\im}{\operatorname{Im}}

\newcommand{\Sp}{\operatorname{Sp}}
\newcommand{\spec}{\operatorname{sp}}
\newcommand{\rk}{\operatorname{rk}}

\renewcommand{\setminus}{\smallsetminus}

\def\F{\mathbb{F}}
\def\K{\mathbb{K}}

\def\N{\mathbb{N}}
\def\Z{\mathbb{Z}}

\renewcommand{\L}{\mathbb{L}}

\def\calA{\mathcal{A}}

\def\calC{\mathcal{C}}

\def\lcro{\mathopen{[\![}}
\def\rcro{\mathclose{]\!]}}

\theoremstyle{definition}
\newtheorem{Def}{Definition}[section]
\newtheorem{Not}[Def]{Notation}

\theoremstyle{plain}
\newtheorem{theo}{Theorem}[section]
\newtheorem{prop}[theo]{Proposition}
\newtheorem{cor}[theo]{Corollary}
\newtheorem{lemma}[theo]{Lemma}

\theoremstyle{plain}

\theoremstyle{remark}
\newtheorem{Rems}{Remarks}
\newtheorem{Rem}[Rems]{Remark}

\title{Products of unipotent elements of index $2$ in orthogonal and symplectic groups}
\author{Cl\'ement de Seguins Pazzis\footnote{Universit\'e de Versailles Saint-Quentin-en-Yvelines, Laboratoire de Math\'ematiques
de Versailles, 45 avenue des Etats-Unis, 78035 Versailles cedex, France}
\footnote{e-mail address: dsp.prof@gmail.com}}

\begin{document}

\thispagestyle{plain}

\maketitle

\begin{abstract}
An automorphism $u$ of a vector space is called unipotent of index $2$ whenever $(u-\id)^2=0$.
Let $b$ be a non-degenerate symmetric or skewsymmetric bilinear form on a vector space $V$ over a field $\F$ of characteristic different from $2$.

Here, we characterize the elements of the isometry group of $b$ that are the product of two unipotent isometries of index $2$. In particular, if $b$ is symplectic we prove that an element of the symplectic group of $b$ is the product of two unipotent isometries of index $2$ if and only if it has no Jordan cell of odd size for the eigenvalue $-1$.
As an application, we prove that every element of a symplectic group is the product of three unipotent elements of index $2$ (and no less in general).

For orthogonal groups, the classification closely matches the classification of sums of two square-zero skewselfadjoint operators that
was obtained in a recent article \cite{dSPsquarezeroquadratic}.
\end{abstract}

\vskip 2mm
\noindent
\emph{AMS Classification:} 15A23, 11E04, 15A21

\vskip 2mm
\noindent
\emph{Keywords:} Orthogonal group, Symplectic group, Decomposition, Unipotent element of index $2$, Wall invariants


\section{Introduction}

\subsection{The problem}

Let $\F$ be a field of characteristic different from $2$, whose group of units we denote by $\F^*$.
Throughout, all the vector spaces under consideration have $\F$ as ground field (unless stated otherwise) and are finite-dimensional.
We use the French notation system for integers: $\N$ denotes the set of all non-negative integers, and $\N^*$ the set of all positive ones.

Throughout, we consider a bilinear form $b : V \times V \rightarrow \F$ on a vector space $V$ and we assume that it is either symmetric
($\forall (x,y)\in V^2, \; b(y,x)=b(x,y)$) or skewsymmetric (i.e. $\forall (x,y) \in V^2, \; b(y,x)=-b(x,y)$).

We also assume that $b$ is non-degenerate, meaning that the (left)-radical $\Rad(b):=\{x \in V : \; b(x,-)=0\}$ of $b$
is zero. We say that $b$ is symplectic when it is non-degenerate and skewsymmetric.

\vskip 3mm
In \cite{Wonenburger}, Wonenburger proved that if $b$ is symmetric (and non-degenerate),
then every element of the orthogonal group $\Ortho(b)$ is the product of two involutions. Her result was later extended by Gow \cite{Gow} to
symplectic groups over fields of characteristic $2$, as well as for orthogonal groups of quadratic forms over such fields.
In symplectic groups over fields of characteristic different from $2$, it is not difficult to find elements that are not products of two involutions, and the
elements that are products of two involutions have been determined by Nielsen (unpublished, see the recent \cite{dSP2invol} for a statement and a proof).

An element $a$ of an $\F$-algebra $\mathcal{A}$ is called \textbf{unipotent of index $2$} whenever $(a-1_\mathcal{A})^2=0_\mathcal{A}$, meaning that $a=1_\mathcal{A}+x$
for some $x \in \mathcal{A}$ such that $x^2=0_\calA$ (i.e.\ $x$ is a square-zero element). Such elements are always invertible.
To abbreviate things, we will call such elements \textbf{$U_2$-elements} of the algebra $\mathcal{A}$.

Remembering that every symplectic transvection is unipotent of index $2$, whereas the reflections are involutions in orthogonal groups, we may view the $U_2$-elements
as the equivalent of the involutions in orthogonal groups, and naturally ask what becomes of Wonenburger and Nielsen's results if
we replace involutions with $U_2$-elements. Of course, since the determinant of a $U_2$-element equals $1$, it is clear that there are isometries that are not the product of two $U_2$-elements (and neither of finitely many of them). One could conjecture that every element of a symplectic group is the product of two $U_2$-elements, however it turns out that this result is false! Indeed,
for the group $\GL(V)$, we have the following known characterization of the elements that are the product of two $U_2$-automorphisms
(see \cite{WangWu} for the field of complex number, and \cite{Bothaunipotent} for the generalization to an arbitrary field):

\begin{theo}[Wang-Wu-Botha theorem]\label{theo:Botha}
Let $u$ be an automorphism of a finite-dimensional vector space $V$ (over a field whose characteristic differs from $2$).
The following conditions are equivalent:
\begin{enumerate}[(i)]
\item $u$ is the product of two $U_2$-automorphisms of $V$;
\item $u$ is similar to its inverse and has no Jordan cell of odd size for the eigenvalue $-1$.
\end{enumerate}
\end{theo}

In a symplectic group, there can be elements that have Jordan cells of size $1$ for the eigenvalue $-1$
(for example $-\id_V$ if $V \neq \{0\}$)!

\vskip 3mm
Remember that the $b$-adjoint of an endomorphism $u$ of $V$ is the unique endomorphism $u^\star$ of $V$ such that
$$\forall (x,y)\in V^2, \; b(u^\star(x),y)=b(x,u(y)).$$
For $u$ to belong to the isometry group of $b$, it is necessary and sufficient that $u$ be invertible and $u^\star=u^{-1}$.
Moreover, if $u=\id_V+a$ for some square-zero endomorphism $a$, then $u^\star=\id_V+a^\star$ and $u^{-1}=\id_V-a$, and hence
$u$ is in the isometry group of $b$ if and only if $a^\star=-a$, i.e.\ $a$ is $b$-skewselfadjoint.
As will turn out, the results are analogous to the ones we have obtained in a recent article \cite{dSPsquarezeroquadratic}, in which
we determined the $b$-skewselfadjoint endomorphisms that are the \emph{sum} of two square-zero $b$-skewselfadjoint endomorphisms.
Our proofs will follow similar patterns although there are subtle differences.

Our results involve the classification of the conjugacy classes in orthogonal and symplectic groups:
the fundamental invariants, which we call the Wall invariants (but which should probably be called Williamson invariants), are recalled in Section \ref{section:Wallinvariants}.

\subsection{The viewpoint of pairs}

It appears that a more efficient viewpoint on our problem is to consider pairs consisting of a form and of an isometry for this form.
More precisely:

\begin{Def}
A \textbf{$1$-isopair} is a pair $(b,u)$ consisting of a non-degenerate symmetric bilinear form $b$ and of an isometry $u \in \Ortho(b)$.

A \textbf{$-1$-isopair} is a pair $(b,u)$ consisting of a symplectic form $b$ and of an isometry $u \in \Sp(b)$.
\end{Def}

Let $\varepsilon \in \{-1,1\}$. Two $\varepsilon$-isopairs $(b,u)$ and $(c,v)$, with underlying vector spaces $U$ and $V$, are called \textbf{isometric} whenever there exists a vector space isomorphism $\varphi : U \overset{\simeq}{\longrightarrow} V$ such that
$$\forall (x,y)\in U^2, \; c(\varphi(x),\varphi(y))=b(x,y) \quad \text{and} \quad \varphi \circ u \circ \varphi^{-1}=v.$$
This defines an equivalence relation on the collection of all $\varepsilon$-isopairs over $\F$.

Next, the orthogonal sum of two $\varepsilon$-isopairs $(b,u)$ and $(b',u')$, with underlying vector spaces $U$ and $U'$,
is defined as the pair $(b,u) \bot (b',u'):=(b \bot b', u \oplus u')$, so that
$$\forall (x,x')\in U \times V, \; \forall (y,y')\in U \times U', \quad
(b \bot b')\bigl((x,x'),(y,y')\bigr)=b(x,y)+b'(x',y')$$
and
$$\forall (x,x')\in U \times U', \; (u \oplus u')(x,x')=\bigl(u(x),u'(x')\bigr).$$
One checks that orthogonal sums are compatible with isometries (i.e.\ replacing one summand with an isometric summand yields an isometric sum).

Another important notion is the one of an induced isopair.
Let $(b,u)$ be an $\varepsilon$-isopair, with underlying vector space $U$, and $V$ be a linear subspace of $U$ that is stable under $u$.
The bilinear form $b$ induces a regular bilinear form $\overline{b}$ on $V/(V \cap V^{\bot_b})$, and it is symmetric if $\varepsilon=1$
and symplectic otherwise. Since $u$ is a $b$-isometry that stabilizes $V$, it also stabilizes $V^{\bot_b}$, and hence also
$V \cap V^{\bot_b}$. Hence, $u$ induces an endomorphism $\overline{u}$ of $V/(V \cap V^{\bot_b})$, and clearly $\overline{u}$
is a $\overline{b}$-isometry. The $\varepsilon$-isopair $(\overline{b},\overline{u})$ is denoted by $(b,u)^V$ and called the isopair
induced by $(b,u)$ on $V$. For example, if $V$ is $b$-regular then so is $V^{\bot_b}$ and
$(b,u) \simeq (b,u)^V \bot (b,u)^{V^{\bot_b}}$.

\begin{Def}
An $\varepsilon$-isopair $(b,u)$ is called \textbf{$U_2$-splittable} when
there exist $U_2$-elements $v$ and $w$ of the isometry group of $b$ such that $u=vw$,
i.e.\ there exist $U_2$-elements $v$ and $w$ of $\GL(V)$, where $V$ denotes the underlying vector space of $(b,u)$, such that $u=vw$, and $(b,v)$ and $(b,w)$ are $\varepsilon$-isopairs.
\end{Def}

\begin{Rem}
If an $\varepsilon$-isopair is $U_2$-splittable, so is any $\varepsilon$-isopair that is isometric to it.
\end{Rem}

\begin{Rem}
Let $(b,u)$ and $(b',u')$ be $U_2$-splittable $\varepsilon$-isopairs.
Choose then $U_2$-elements $s_1,s_2$ in the isometry group of $b$, and $U_2$-elements $s'_1,s'_2$ in the isometry group of $b'$, such that
$u=s_1s_2$ and $u'=s'_1s'_2$. Then, with $S_1:=s_1 \oplus s'_1$ and $S_2:=s_2 \oplus s'_2$, one sees that $u \oplus u'=S_1S_2$ and that
$S_1$ and $S_2$ are $U_2$-elements and belong to the isometry group of $b \bot b'$.
Hence, the $\varepsilon$-isopair $(b,u) \bot (b',u')$ is $U_2$-splittable.
\end{Rem}

The converse fails. In theory, it is possible that none of $(b,u)$ and $(b',u')$ is $U_2$-splittable, but
$(b,u) \bot (b',u')$ is.

\begin{Rem}\label{remark:inducedpairs}
Let $(b,u)$ be a $U_2$-splittable $\varepsilon$-isopair, choose a corresponding splitting $u=u_1u_2$, and let
$V$ be a subspace of the underlying space of $(b,u)$ that is stable under both $u_1$ and $u_2$.
Then clearly the $\varepsilon$-isopairs $(b,u_1)^V$ and $(b,u_2)^V$ can be used to see that $(b,u)^V$ is $U_2$-splittable.
\end{Rem}

\begin{Def}
An $\varepsilon$-isopair is called \textbf{trivial} if the underlying vector space is zero.

An $\varepsilon$-isopair is called \textbf{indecomposable} when it is nontrivial
and it is not isometric to the orthogonal direct sum of two nontrivial isopairs
(in other words, such a pair $(b,u)$, over a space $V$, is indecomposable when there is no $b$-orthogonal decomposition
$V=V_1 \overset{\bot_b}{\oplus} V_2$ in which $V_1$ and $V_2$ are nonzero and stable under $u$).
\end{Def}

\subsection{A review of conjugacy classes and extensions of bilinear forms}\label{section:extension}

We denote by $\Irr(\F)$ the set of all irreducible polynomials $p \in \F[t]$ such that $p \neq t$ (beware that
in \cite{dSPsquarezeroquadratic} the notation $\Irr(\F)$ was used for a different set).

Let $V$ be a vector space over $\F$.
The classification of similarity classes in the algebra $\End(V)$ is well known: here, we will use the viewpoint of the primary invariants.
Every endomorphism $u$ of $V$ is the direct sum of cyclic endomorphisms whose minimal polynomials are powers of
monic irreducible polynomials: those minimal polynomials are uniquely determined up to permutation, and they are called the
primary invariants of $u$. For $p \in \Irr(\F) \cup \{t\}$ and $r \geq 1$, we denote by $n_{p,r}(u)$ the number of summands with minimal polynomial
$p^r$ in such a decomposition (it is the \textbf{Jordan number} of $u$ attached to the pair $(p,r)$).
The similarity class of $u$ is then determined by the data of the family $(n_{p,r}(u))_{p \in \Irr(\F)\cup \{t\}, r \geq 1}$.

Now, let $u$ be a $b$-isometry, and consider the group $G$ of all $b$-isometries.
For $u$ to be $U_2$-splittable in $G$, it is necessary that $u$ be similar to $u^{-1}$ (this is not obvious at all, see Theorem \ref{theo:Botha} nevertheless).
However, this is automatic for a $b$-isometry because classically $u^\star$ is similar to $u$!

Given a monic polynomial $p(t)$ of degree $d$ such that $p(0) \neq 0$, we denote by
$$p^\sharp(t):=p(0)^{-1}t^d p(1/t)$$
the reciprocal polynomial of $p$ (note that $(p^\sharp)^\sharp=p$), and we say that $p$ is a \textbf{palindromial}
whenever $p^\sharp=p$. In that case and if $p$ is irreducible and $p \neq t\pm 1$, it turns out that $p(0)=1$ and that the degree $d$ of $p$ is even.
Then, an automorphism $u$ of a vector space $V$ is similar to its inverse if and only if, for all
$p \in \Irr(\F)$ such that $p^\sharp \neq p$, and for all $r \geq 1$, one has $n_{p,r}(u)=n_{p^\sharp,r}(u)$.

Now, let $p \in \Irr(\F)$ be an irreducible palindromial that is different from $t\pm 1$. Denoting by $2d$ the degree of $p$, one checks that
$p(t)=t^d m(t+t^{-1})$ for a unique $m \in \Irr(\F) \setminus \{t\pm 2\}$ of degree $d$.
We consider the ring $\F[t,t^{-1}]$ of Laurent polynomials, equipped with the involution that takes $t$ to $t^{-1}$.
The subring of selfadjoint elements is $\F[t+t^{-1}]$.
The ideal generated by $p$ is invariant under the involution under consideration, the quotient field
$\L:=\F[t,t^{-1}]/(p)$ is naturally isomorphic to the residue field $\F[t]/(p)$, and we denote by
$\lambda \mapsto \lambda^\bullet$ the induced involution of $\L$.
In $\L$, the subfield $\K$ of selfadjoint elements is $\F[\overline{t}+\overline{t}^{-1}]$, where $\overline{t}$ denotes the class of the indeterminate
$t$ modulo $(p)$, and in $\L$ the minimal polynomial of $\overline{t}+\overline{t}^{-1}$ is $m$.
Hence, the subfield $\K$ is isomorphic to $\F[t]/(m)$
through the isomorphism $\psi$ that takes the class $\overline{t}+\overline{t}^{-1}$ to the class of $t$ modulo $m$.
We consider the $\F$-linear form $e_m : \F[t]/(m) \rightarrow \F$ that takes the class of $1$ to $1$, and the class of $t^k$ to $0$ for all
$k \in \lcro 1,d-1\rcro$,
and finally we set
$$f_p : \lambda \in \L \mapsto e_m(\psi(\lambda+\lambda^\bullet)) \in \F.$$
Hence, $f_p$ is a nonzero linear form on the $\F$-vector space $\L$, and $f_p(\lambda^\bullet)=f_p(\lambda)$ for all $\lambda \in \L$.

The next step is a general construction: let $V$ be an $\L$-vector space, and
$B : V^2 \rightarrow \F$ be an $\F$-bilinear form such that
$$\forall (x,y)\in V^2, \; \forall \lambda \in \mathbb{L}, \; B(\lambda^\bullet x,y)=B(x,\lambda y).$$
For all $(x,y)\in V^2$, the mapping $\lambda \mapsto B(x,\lambda y)$ is an $\F$-linear form on
$\L$ and hence it reads $\lambda \mapsto f_p(\lambda B^\L(x,y))$ for a unique $B^\L(x,y)\in \L$.
This yields a mapping $B^\L : V^2 \rightarrow \L$ and one checks that it is $\F$-bilinear and even right-$\L$-linear.
Moreover, if $B$ is non-degenerate then so is $B^\L$.
Now, assume that $B$ is symmetric. Let $(x,y) \in V^2$. Then, for all $\lambda \in \F$,
$$B(y,\lambda x)=B(\lambda^\bullet y,x)=B(x,\lambda^\bullet y)=f_p(\lambda^\bullet B^\L(x,y))=f_p\bigl(\lambda B^\L(x,y)^\bullet\bigr)$$
and hence $B^{\mathbb{L}}(y,x)=B^\L(x,y)^\bullet$. Likewise, if $B$ is skewsymmetric one finds
$B^{\mathbb{L}}(y,x)=-B^\L(x,y)^\bullet$. To sum up:
\begin{itemize}
\item If $B$ is symmetric and non-degenerate, then $B^{\mathbb{L}}$ is a non-degenerate Hermitian form on the $\L$-vector space $V$.
\item If $B$ is symplectic then $B^{\mathbb{L}}$ is a non-degenerate skew-Hermitian form on the $\L$-vector space $V$.
\end{itemize}
Remember finally that skew-Hermitian forms are in one-to-one correspondence with Hermitian forms: by taking an element
$\eta \in \L^*$ such that $\eta^\bullet =-\eta$ (such an element always exists), the Hermitian form $h$ gives rise to the
skew-Hermitian form $\eta h$ and vice versa.

\subsection{A review of conjugacy classes in orthogonal and symplectic groups}\label{section:Wallinvariants}

Let $(b,u)$ be an $\varepsilon$-isopair, and let $r \geq 1$ be a positive integer.
Let $p \in \Irr(\F)$ be an irreducible \emph{palindromial} of degree $2d$.

Set $v:=u+u^{-1}=u+u^\star$, which is $b$-selfadjoint.
We denote by $V_{p,r}$ the cokernel of the (injective) linear map
$$\Ker m(v)^{r+1}/\Ker m(v)^r \longrightarrow \Ker m(v)^{r}/\Ker m(v)^{r-1}$$
induced by $m(v)$. This cokernel is naturally identified with the quotient space
$$V_{p,r}:=\Ker m(v)^r/\bigl(\Ker m(v)^{r-1}+(\im m(v) \cap \Ker m(v)^r)\bigr).$$
We consider the bilinear form
$$b_{p,r} : (x,y) \mapsto b\bigl(x,m(v)^{r-1}[y]\bigr)$$
on $\Ker m(v)^{r}$.
Noting that $m(v)$ is $b$-selfadjoint, we obtain that $b_{p,r}$ is symmetric if $\varepsilon=1$,
skewsymmetric if $\varepsilon=-1$.
The radical of $b_{p,r}$ is the intersection of $\Ker m(v)^r$
with the inverse image of $(\Ker m(v)^r)^{\bot_b}=\im m(v)^r$ under $m(v)^{r-1}$, and one easily
checks that this inverse image equals $\Ker m(v)^{r-1}+(\Ker m(v)^r \cap \im m(v))$.
Hence, $b_{p,r}$ induces a non-degenerate $\F$-bilinear form $\overline{b_{p,r}}$ on $V_{p,r}$,
a form that is symmetric if $\varepsilon=1$, skewsymmetric if $\varepsilon=-1$.
The quotient spaces $\Ker m(v)^{r+1}/\Ker m(v)^r$ and $\Ker m(v)^{r}/\Ker m(v)^{r-1}$ are naturally seen as vector spaces over $\L:=\F[t,t^{-1}]/(p)$, and hence so are the said cokernel.
Since $u$ is a $b$-isometry it turns out that
$\overline{b_{p,r}}(x,\lambda y)=\overline{b_{p,r}}(\lambda^\bullet x,y)$ for all
$\lambda \in \L$ and all $(x,y)\in (V_{p,r})^2$.
Hence, by using the extension process described in Section \ref{section:extension}, we recover a form
$\overline{b_{p,r}}^\L$ on the $\L$-vector space $V_{p,r}$, either Hermitian if $b$ is symmetric, or skew-Hermitian if $b$ is symplectic.
In any case, we denote by $(b,u)_{p,r}$ this form and call it the \textbf{Hermitian Wall invariant} of $(b,u)$ attached to $(p,r)$.

The following facts are easily checked:
\begin{itemize}
\item Given an $\varepsilon$-isopair $(b',u')$ that is isometric to $(b,u)$, the Hermitian or skew-Hermitian forms
$(b,u)_{p,r}$ and $(b',u')_{p,r}$ are equivalent;
\item Given another $\varepsilon$-isopair $(b',u')$, one has
$((b,u) \bot (b',u'))_{p,r} \simeq (b,u)_{p,r} \bot (b',u')_{p,r}$.
\end{itemize}

There are two additional sets of invariants, and here we must differentiate more profoundly between symmetric and
symplectic forms.
Assume first that $\varepsilon=1$, and let $r=2k+1$ be an odd positive integer.
Again, we set $v:=u+u^{-1}$ and we note that $(v-2\id_V)^k=(-1)^k (u-\id_V)^k (u^{-1}-\id_V)^k$.
Like in the above, the symmetric bilinear form
$(x,y) \mapsto b(x,(v-2\id_V)^k(y))$ induces a non-degenerate symmetric bilinear form
$(b,u)_{t-1,2k+1}$ on $\Ker (u-\id)^r/\bigl(\Ker (u-\id)^{r-1}+(\im (u-\id) \cap \Ker (u-\id)^r)\bigr)$.
Likewise,
$(x,y) \mapsto b(x,(v+2\id_V)^k(y))$ induces a non-degenerate symmetric bilinear form
$(b,u)_{t+1,2k+1}$ on $\Ker (u+\id)^r/\bigl(\Ker (u+\id)^{r-1}+(\im (u+\id) \cap \Ker (u+\id)^r)\bigr)$.
These are the \textbf{quadratic Wall invariants} of $(b,u)$.

The classification of $1$-isopairs is then given in the next theorem.

\begin{theo}[Wall's theorem for orthogonal groups, see \cite{Wall}]
Let $(b,u)$ and $(b',u')$ be $1$-isopairs.
For $(b,u)$ to be isometric to $(b',u')$, it is necessary and sufficient that all the following conditions hold:
\begin{enumerate}[(i)]
\item For all $p \in \Irr(\F)$ which is not a palindromial, and all $r \geq 1$, one has $n_{p,r}(u)=n_{p,r}(u')$.
\item For all $p \in \Irr(\F) \setminus \{t+1,t-1\}$ such that $p=p^\sharp$, and all $r \geq 1$,
the Hermitian forms $(b,u)_{p,r}$ and $(b',u')_{p,r}$ are equivalent.
\item For every $\eta \in \{-1,1\}$ and every odd integer $r \geq 1$, the non-degenerate symmetric bilinear forms
$(b,u)_{t-\eta,r}$ and $(b',u')_{t-\eta,r}$ are equivalent.
\end{enumerate}
\end{theo}

Next, we complete the classification of $-1$-isopairs.
Let $(b,u)$ be such an isopair, and set $v:=u+u^{-1}$.
Let $r=2k+2$ be an even positive integer.
The bilinear form $(x,y) \mapsto b(x,(u-u^{-1})(v-2\,\id_V)^k(y))$ is symmetric.
Noting that $(u-u^{-1})(v-2\id_V)^k=(-1)^k u^{-(k+1)}(u+\id_V) (u-\id_V)^{2k+1}$, one finds that
this bilinear form induces a non-degenerate symmetric bilinear form on the
quotient space $\Ker (u-\id_V)^r/\bigl(\Ker (u-\id_V)^{r-1}+(\im (u-\id_V) \cap \Ker (u-\id_V)^r)\bigr)$,
and we denote this form by $(b,u)_{t-1,r}$.
Likewise, $(x,y) \mapsto b(x,(u-u^{-1})(v+2\id_V)^k(y))$
 induces a non-degenerate symmetric bilinear form on the
quotient space $\Ker (u+\id_V)^r/\bigl(\Ker (u+\id_V)^{r-1}+(\im (u+\id_V) \cap \Ker (u+\id_V)^r)\bigr)$,
and we denote this form by $(b,u)_{t+1,r}$.
These are the \textbf{quadratic Wall invariants} of $(b,u)$.

\begin{theo}[Wall's theorem for symplectic groups, see \cite{Wall}]
Let $(b,u)$ and $(b',u')$ be two $-1$-isopairs.
For $(b,u)$ to be isometric to $(b',u')$, it is necessary and sufficient that:
\begin{enumerate}[(i)]
\item For all $p \in \Irr(\F)$ which is not a palindromial, and all $r \geq 1$, one has $n_{p,r}(u)=n_{p,r}(u')$.
\item For all $p \in \Irr(\F) \setminus \{t+1,t-1\}$ such that $p=p^\sharp$, and all $r \geq 1$,
the skew-Hermitian forms $(b,u)_{p,r}$ and $(b',u')_{p,r}$ are equivalent.
\item For every $\eta \in \{-1,1\}$ and every even integer $r \geq 2$, the non-degenerate symmetric bilinear forms
$(b,u)_{t-\eta,r}$ and $(b',u')_{t-\eta,r}$ are equivalent.
\end{enumerate}
\end{theo}

We finish with a different viewpoint, which is especially useful for our problem: a description of indecomposable pairs.

\begin{theo}\label{theo:indecomposable}
Every indecomposable $1$-isopair (respectively, $-1$-isopair) $(b,u)$ satisfies one of the following properties:
\begin{itemize}
\item $u$ is cyclic and its minimal polynomial equals $p^r$ for some palindromial $p \in \Irr(\F) \setminus \{t \pm 1\}$
and some $r \geq 1$;
\item $u$ has exactly two primary invariants, equal to $p^r$ and $(p^\sharp)^r$ for some
$p \in \Irr(\F)$ such that $p \neq p^\sharp$, and some $r \geq 1$;
\item $u$ is cyclic and its minimal polynomial equals $(t-\eta)^r$ for some odd (respectively, even) integer $r$ and some $\eta =\pm 1$;
\item $u$ has exactly two primary invariants, both equal to $(t-\eta)^r$ for some even (respectively, odd) integer $r$ and some $\eta =\pm 1$.
\end{itemize}
\end{theo}

\subsection{Main results}\label{section:results}

We start with the case of symplectic groups, which is the easier:

\begin{theo}\label{theo:symplectic}
Let $b$ be a symplectic form over a field of characteristic different from $2$.
Let $u \in \Sp(b)$. The following conditions are equivalent:
\begin{enumerate}[(i)]
\item $u$ is the product of two $U_2$-elements of $\GL(V)$;
\item $u$ is the product of two $U_2$-elements of $\Sp(b)$;
\item $u$ has no Jordan cell of odd size for the eigenvalue $-1$.
\end{enumerate}
\end{theo}

As a corollary, we obtain a full solution to the length problem:

\begin{theo}\label{theo:symplectic3}
Let $b$ be a symplectic form over a field of characteristic different from $2$.
Then every element of $\Sp(b)$ is the product of three unipotent elements of index $2$.
\end{theo}

In orthogonal groups, there are much more restrictive conditions for an isometry to be the product of two $U_2$-isometries.
We will split the study into two cases: the one of $1$-isopairs $(b,u)$ with $u$ unipotent (i.e.\ $u-\id$ is nilpotent) and the one of $1$-isopairs $(b,u)$ with $1 \not\in \spec(u)$ (i.e. $u-\id$ is invertible).

To see that such a reduction is relevant, we start with a simple algebraic lemma.

\begin{lemma}[Commutation lemma]\label{lemma:commutation}
Let $u$ be an automorphism of a vector space $V$, and $u_1,u_2$ be $U_2$-automorphisms of $V$ such that $u=u_1u_2$.
Then :
\begin{enumerate}[(i)]
\item $u_1(u-\id)=(u-\id)u_2$ and $u_1^{-1}(u-\id)=(u-\id) u_2^{-1}$.
\item Both $u_1$ and $u_2$ stabilize $\Ker(u-\id)$ and $\im(u-\id)$.
\item Both $u_1$ and $u_2$ commute with $u+u^{-1}$.
\end{enumerate}
\end{lemma}

\begin{proof}
We write $u_1=\id+a_1$ and $u_2=\id+a_2$, so that $a_1^2=a_2^2=0$.
Note that $u_1^{-1}=\id-a_1$ and $u_2^{-1}=\id-a_2$.
Noting that $u-\id=a_1+a_2+a_1a_2$ we find $a_1(u-\id)=a_1a_2=(u-\id)a_2$,
which yields point (i).

By point (i), both $u_2$ and $u_2^{-1}$ stabilize $\Ker(u-\id)$, and both $u_1$ and $u_1^{-1}$ stabilize $\im(u-\id)$.
It follows that $u_2=u_1^{-1} u$ stabilizes $\im(u-\id)$ and that $u_1=u u_2^{-1}$ stabilizes $\Ker(u-\id)$.

Noting that $u+u^{-1}=2\id+a_1a_2+a_2a_1$, we have $a_1(u+u^{-1})=2a_1+a_1a_2a_1=(u+u^{-1})a_1$.
Hence $u_1=\id+a_1$ commutes with $u+u^{-1}$.
Finally, $u_2=u_1^{-1} u$ also commutes with $u+u^{-1}$.
\end{proof}

\begin{cor}[Stabilization Lemma]\label{cor:stabilizationlemma}
Let $u$ be an automorphism of a vector space $V$, and $u_1,u_2$ be $U_2$-automorphisms of $V$ such that $u=u_1u_2$.
Then $\im (u-\id)^k$ and $\Ker (u-\id)^k$ are stable under $u_1$ and $u_2$ for every integer $k \geq 0$.
\end{cor}

Note that the stabilization of $\Ker (u-\id)^k$ can be derived from corollary 4.5 in \cite{dSPprodexceptional}.

\begin{proof}
Note first that $u+u^{-1}-2\id=u^{-1}(u-\id)^2$. Let $k \geq 0$. Then
$u^{-k} (u-\id)^{2k}=(u+u^{-1}-2\id)^k=(u-\id)^{2k}\,u^{-k}$, and hence
$\im (u-\id)^{2k}=\im (u+u^{-1}-2\id)^k$ and $\Ker (u-\id)^{2k}=\Ker (u+u^{-1}-2\id)^k$.
Since $u_1$ and $u_2$ commute with $u+u^{-1}$, we deduce that they stabilize $\im (u-\id)^{2k}$ and $\Ker (u-\id)^{2k}$.

Next, we already know from point (ii) of Lemma \ref{lemma:commutation} that $u_1$ and $u_2$
stabilize $\im(u-\id)$, and of course so does $u$. Their respective restrictions $u'_1,u'_2,u'$ to $\im u$ satisfy
$u'=u'_1u'_2$, and $u'_1$ and $u'_2$ are $U_2$-automorphisms of $\im(u-\id)$. By the first step, we obtain that, for every integer $k \geq 0$,
the endomorphisms $u'_1$ and $u'_2$ stabilize $\im(u'-\id)^{2k}=\im(u-\id)^{2k+1}$, which means that
$u_1$ and $u_2$ stabilize $\im(u-\id)^{2k+1}$.

Finally, $u_1$, $u_2$ and $u$ all stabilize $\Ker(u-\id)$ and hence they induce respective endomorphisms
$\overline{u_1}$, $\overline{u_2}$ and $\overline{u}$ of the quotient vector space $V/\Ker(u-\id)$,
with $\overline{u}=\overline{u_1}\,\overline{u_2}$. Note that $\overline{u_1}$ and $\overline{u_2}$ are $U_2$-automorphisms.
Letting $k \geq 0$ be a positive integer, we obtain that $\overline{u_1}$ and $\overline{u_2}$ stabilize
$\Ker(\overline{u}-\id)^{2k}=\Ker(u-\id)^{2k+1}/\Ker(u-\id)$, which yields that $u_1$ and $u_2$ stabilize $\Ker(u-\id)^{2k+1}$.
\end{proof}

Now, consider a $U_2$-splittable $1$-isopair $(b,u)$, together with
a pair $(u_1,u_2)$ of $U_2$-isometries such that $u=u_1u_2$. Denote by $V$ the underlying vector space.
We introduce the Fitting decomposition of $u-\id$, consisting of
$$\Co(u-\id):=\underset{n \in \N}{\bigcap} \im (u-\id)^n \quad \text{and} \quad 
\Nil(u-\id):=\underset{n \in \N}{\bigcup} \Ker (u-\id)^n.$$
Remember that $V=\Co(u-\id) \oplus \Nil(u-\id)$ and that $u$ induces an automorphism of $\Co(u-\id)$
of which $1$ is no eigenvalue, and a unipotent automorphism of $\Nil(u-\id)$.
Since $u$ is a $b$-isometry, we have, for all $k \geq 0$,
$$(\im (u-\id)^k)^{\bot_b}=\Ker ((u-\id)^k)^\star=\Ker (u^{-1}-\id)^k=\Ker (u-\id)^k,$$
and hence $\Nil(u-\id)$ is $b$-orthogonal to $\Co(u-\id)$.
It follows from the Stabilization Lemma
that both $u_1$ and $u_2$ stabilize $\Co(u-\id)$ and $\Nil(u-\id)$.
By Remark \ref{remark:inducedpairs}, this yields that the induced pairs $(b,u)^{\Nil(u-\id)}$ and $(b,u)^{\Co(u-\id)}$ are $U_2$-splittable.

Conversely, if we start from a $1$-isopair $(b,u)$ such that both induced pairs $(b,u)^{\Nil(u-\id)}$ and $(b,u)^{\Co(u-\id)}$ are $U_2$-splittable, then so is
$(b,u)$ since it is isometric to $(b,u)^{\Nil(u-\id)} \bot (b,u)^{\Co(u-\id)}$.

\begin{prop}\label{prop:reductionto2cases}
Let $(b,u)$ be a $1$-isopair. For $(b,u)$ to be $U_2$-splittable,
it is necessary and sufficient that both $(b,u)^{\Co(u-\id)}$ and $(b,u)^{\Nil(u-\id)}$ be $U_2$-splittable.
\end{prop}

This proposition allows us to separate the study into the one of the $1$-isopairs $(b,u)$ in which $u-\id$ is invertible, and the one
of the $1$-isopairs $(b,u)$ in which $u$ is unipotent.
We start with the former, as it is substantially easier.

\begin{theo}\label{theo:orthogonal1noeigenvalue}
Let $(b,u)$ be a $1$-isopair. Assume that $u-\id$ is invertible.
For $(b,u)$ to be $U_2$-splittable, it is necessary and sufficient that all the following conditions hold:
\begin{enumerate}[(i)]
\item $u$ has no Jordan cell of odd size for the eigenvalue $-1$.
\item All the Jordan numbers of $u$ are even.
\item All the Wall invariants of $(b,u)$ are hyperbolic.
\end{enumerate}
\end{theo}

We now turn to unipotent automorphisms, for which the condition on the (quadratic) Wall invariants is far more technical.

\begin{Def}
A $1$-isopair $(b,u)$ is called \textbf{unipotent} whenever $u$ is unipotent.
\end{Def}

Let $B$ be a non-degenerate symmetric bilinear form (on a finite-dimensional vector space $V$). Remember that
$B$ is the orthogonal sum of a hyperbolic symmetric bilinear form and of a
nonisotropic symmetric bilinear form; the rank of the hyperbolic form equals $2\nu(B)$, where $\nu(B)$
is the \textbf{Witt index} of $B$, i.e.\ the greatest dimension for a totally $B$-isotropic subspace of $V$; the
equivalence ``class" of the nonisotropic form depends only on the one of $B$, and we say that this nonisotropic form is a
\textbf{nonisotropic part} of $B$.
Finally, a \textbf{subform} of $B$ is simply the restriction of $B$ to $W^2$ for some linear subspace $W$ of $V$.

\begin{Def}
Let $B$ and $B'$ be non-degenerate symmetric bilinear forms (on finite-dimensional vector spaces over $\F$).
The following conditions are equivalent:
\begin{itemize}
\item Every nonisotropic part of $B'$ is equivalent to a subform of $-B$;
\item One has $\nu(B')+\nu(B' \bot B) \geq \rk(B')$.
\end{itemize}
When they are satisfied, we say that $B$ \textbf{Witt-simplifies} $B'$.
\end{Def}

See section 1.5 of \cite{dSPsquarezeroquadratic} for the proof that these conditions are equivalent.

\begin{theo}\label{theo:orthogonalunipotent}
Let $(b,u)$ be a unipotent $1$-isopair.
For an integer $k \geq 0$, denote by $B_k$ the quadratic Wall invariant $(b,u)_{t-1,2k+1}$.
The following conditions are equivalent:
\begin{enumerate}[(i)]
\item $u$ is the product of two $U_2$-elements of $\Ortho(b)$.
\item For every integer $k \geq 0$, $(-1)^k B_k$ Witt-simplifies $\underset{i >k}{\bot}(-1)^i B_i$.
\end{enumerate}
\end{theo}

By combining the previous two theorems thanks to Proposition \ref{prop:reductionto2cases}, we obtain the complete classification of
the elements of an orthogonal group that are products of two $U_2$-isometries.

\begin{theo}\label{theo:orthogonal}
Let $(b,u)$ be a $1$-isopair.
For $k \geq 0$, denote by $B_k$ the quadratic Wall invariant $(b,u)_{t-1,2k+1}$.
For $u$ to be the product of two unipotent elements of index $2$ in $\Ortho(b)$, it is necessary and sufficient that all the following
conditions hold:
\begin{enumerate}[(i)]
\item $u$ has no Jordan cell of odd size for the eigenvalue $-1$.
\item The Jordan number $n_{p,r}(u)$ is even for all $p \in \Irr(\F) \setminus \{t-1\}$ and all $r \geq 1$.
\item All the Hermitian Wall invariants of $(b,u)$ are hyperbolic.
\item For every integer $k \geq 0$, the non-degenerate symmetric bilinear form
$(-1)^k B_k$ Witt-simplifies $\underset{i >k}{\bot}(-1)^i B_i$.
\end{enumerate}
\end{theo}

\subsection{Structure of the article}

The remainder of the article is laid out as follows.
In Section \ref{section:construction}, we develop two main ways to construct interesting elements of orthogonal and symplectic groups.
The simpler is the first, which we call the hyperbolic/symplectic extensions of an automorphism.
A more subtle construction is actually required: the boxed-product of a pair of symmetric/skewsymmetric bilinear forms over the same vector space. It turns out (Proposition \ref{prop:caracboxedprod}) that every $U_2$-splittable pair $(b,u)$, where $u-\id$ is invertible,
is isometric to such a boxed-product. And conversely, every boxed-product is $U_2$-splittable.
This motivates a thorough study of boxed-products and their invariants (Jordan numbers and Wall invariants), which is performed in Section \ref{section:invariantsboxedproduct}.

Using these two constructions, we complete the study for symplectic groups in Section \ref{section:symplectic}, and the case of elements of orthogonal groups that do not have $1$ as eigenvalue, in Section \ref{section:orthogonalnonunipotent}. In Section \ref{section:symplectic}, we also obtain that every element of a symplectic group is the product of three $U_2$-elements.

However, for orthogonal groups the case of unipotent automorphisms is very difficult, and none of the standard constructions
is sufficient to completely account for it. This difficult case is dealt with in Section \ref{section:unipotent} by
adapting the proof of the corresponding theorem for the decomposition of a skewselfadjoint endomorphism into the sum of two square-zero
skewselfadjoint endomorphisms (see section 5 of \cite{dSPsquarezeroquadratic}).

\section{Two key constructions}\label{section:construction}

\subsection{Review of duality theory}

Let us recall some basic notation and facts on dual spaces. Let $V$ be a finite-dimensional vector space
over $\F$. We denote by $V^\star:=\Hom(V,\F)$ its dual space. It has the same dimension as $V$.
The transpose of an endomorphism $u$ of $V$ is defined as
$$u^t : \varphi \in V^\star \mapsto u \circ \varphi \in V^\star.$$
A classical consequence of the Frobenius normal form is that $u^t$ is similar to $u$.

\subsection{Hyperbolic and symplectic extensions of an automorphism}\label{expansionsection}

Let $V$ be a finite-dimensional vector space, and let $\varepsilon \in \{-1,1\}$.
On the product space $V \times V^\star$, we consider the bilinear form
$$H_V^\varepsilon : \begin{cases}
(V \times V^\star)^2 & \longrightarrow \F \\
((x,\varphi),(y,\psi)) & \longmapsto \varphi(y)+\varepsilon \psi(x).
\end{cases}$$
One checks that $H_V^1$ is a non-degenerate hyperbolic symmetric bilinear form, whereas $H_V^{-1}$ is a symplectic form.

Let $u \in \GL(V)$.
We consider the endomorphism
$$h(u) : \begin{cases}
V \times V^\star & \longrightarrow V \times V^\star \\
(x,\varphi) & \longmapsto \bigl(u(x), (u^{-1})^t(\varphi)\bigr).
\end{cases}$$
One checks that $(H_V^\varepsilon,h(u))$ is an $\varepsilon$-isopair, and we denote it by $H_\varepsilon(u)$ (see section 2.1 of \cite{dSP2invol} for details).

Let $v \in \GL(V)$ and $w \in \GL(W)$ be similar automorphisms. It was proved in \cite{dSP2invol} that
$$H_\varepsilon(v) \simeq H_\varepsilon(w).$$
Finally, let $u_1 \in \GL(V_1)$ and $u_2 \in \GL(V_2)$ be automorphisms.
Then, it was proved in \cite{dSP2invol} that
$$H_\varepsilon(u_1 \oplus u_2) \simeq H_\varepsilon(u_1) \bot  H_\varepsilon(u_2).$$
The similarity ``class" of the automorphism $h(u)$ is easily deduced from the one of $u$ because $u^t$ is similar to $u$.
Simply, $n_{p,k}(h(u))=n_{p,k}(u) +n_{p^\sharp,k}(u)$ for all $p \in \Irr(\F)$ and all $k \geq 1$.

Note that $h(uv)=h(u)h(v)$ for all $u,v$ in $\GL(V)$, and that if $u$ is unipotent with index $2$ then so is $h(u)$.
Thus, we deduce the following result:

\begin{prop}\label{prop:hyperbolicexpansionU2split}
Let $u \in \GL(V)$ be the product of two $U_2$-elements of $\GL(V)$. Then
$h(u)$ is the product of two $U_2$-elements, both in $\Ortho(H_V^1)$ and in $\Sp(H_V^{-1})$.
\end{prop}

As in \cite{dSP2invol} (see lemma 2.4 there), we obtain:

\begin{prop}\label{prop:Wallinvariantshyperbolic}
Let $u \in \GL(V)$ and $\varepsilon \in \{-1,1\}$. Then all the Wall invariants of $H_\varepsilon(u)$
are hyperbolic.
\end{prop}

\subsection{Boxed-products of two bilinear forms}\label{section:crossedproduct}

Let $V$ be an arbitrary finite-dimensional vector space. Let $\varepsilon\in \{-1,1\}$.

Let $b$ be a bilinear form on $V$ (at this point we do not assume that it is symmetric or skewsymmetric, and $b$ can very well be degenerate).
We consider the associated linear map $L_b : x \in V \mapsto b(-,x) \in V^\star$, and we define the endomorphism
$$v_b : (x,\varphi) \in V \times V^\star \mapsto \bigl(0,L_b(x)\bigr) \in V \times V^\star,$$
which is obviously of square zero.
For all $(x,\varphi)$ and $(y,\psi)$ in $V \times V^\star$,
$$H_V^\varepsilon \bigl(v_b(x,\varphi),(y,\psi)\bigr)=L_b(x)[y]=b(y,x).$$
Hence:
\begin{itemize}
\item If $\varepsilon=1$ then $v_b$ is $H_V^\varepsilon$-skewselfadjoint if and only $b$ is skewsymmetric.
\item If $\varepsilon=-1$ then $v_b$ is $H_V^\varepsilon$-skewselfadjoint if and only $b$ is symmetric.
\end{itemize}
Assume further that $b$ is non-degenerate. Then $L_b$ is an isomorphism and we can consider the endomorphism,
$$w_b : (x,\varphi) \in V \times V^\star \mapsto (L_b^{-1}(\varphi),0) \in V \times V^\star,$$
again obviously of square zero.
Let $(x,\varphi)$ and $(y,\psi)$ in $V \times V^\star$. Setting $x',y'$ in $V$ such that $L_b(x')=\varphi$
and $L_b(y')=\psi$, we find that
$$H_V^\varepsilon \bigl(w_b(x,\varphi),(y,\psi)\bigr)=\varepsilon\, \psi(x')=\varepsilon\, b(x',y').$$
Again, we find that:
\begin{itemize}
\item If $\varepsilon=1$ then $w_b$ is $H_V^\varepsilon$-skewselfadjoint if and only if $b$ is skewsymmetric.
\item If $\varepsilon=-1$ then $w_b$ is $H_V^\varepsilon$-skewselfadjoint if and only if $b$ is symmetric.
\end{itemize}

\begin{Def}
Let $b,c$ be two bilinear forms on a finite-dimensional vector space $V$, with $b$ non-degenerate.
We define their \textbf{boxed-product} as the automorphism
$$b \boxtimes c:=(\id+v_c)(\id+w_b) \in \GL(V \times V^\star).$$
\end{Def}

Next, if $c$ is non-degenerate then one checks that $b \boxtimes c -\id$ is invertible
(indeed, letting $(x,\varphi)$ belong to the kernel of $b \boxtimes c -\id$, we would have $(L_b^{-1}(\varphi),L_c(x)+(L_cL_b^{-1})(\varphi))=(0,0)$,
leading to $\varphi=0$ and then to $x=0$).
Conversely, if $c$ is degenerate then we can take $x \in \Ker L_c \setminus\{0\}$, and then $(x,0)$ is a nonzero vector in the kernel of $b \boxtimes c -\id$.

Finally, note that if $w_b$ and $v_c$ are $H_V^\varepsilon$-skewselfadjoint, then $b \boxtimes c$
is obviously the product of two $U_2$-elements in the isometry group of $H_V^\varepsilon$.

Here is a first conclusion:

\begin{prop}
Let $\varepsilon \in \{-1,1\}$. Let $b,c$ be bilinear forms on the same vector space $V$, with $b$ non-degenerate.
Assume that $b$ and $c$ are both symmetric if $\varepsilon=-1$, and both skewsymmetric otherwise.
Then $(H_V^\varepsilon, b \boxtimes c)$ is a $U_2$-splittable $\varepsilon$-isopair.
Moreover, $b \boxtimes c-\id$ is invertible if and only if $c$ is non-degenerate.
\end{prop}

Now we prove a converse statement:

\begin{prop}\label{prop:caracboxedprod}
Let $\varepsilon\in \{-1,1\}$, and let $(b,u)$ be a $U_2$-splittable $\varepsilon$-isopair such that $u-\id$ is invertible.
Then there exists a vector space $W$ together with a pair $(B,C)$ of non-degenerate bilinear forms on $W$,
both symmetric if $\varepsilon=-1$, both skewsymmetric otherwise, such that
$$(b,u) \simeq (H_W^{\varepsilon},B \boxtimes C).$$
\end{prop}

\begin{proof}
Denote by $V$ the underlying vector space of $(b,u)$.
Let $u_1$ and $u_2$ be $U_2$-elements of the isometry group of $b$ such that $u=u_2u_1$.
Then $u_1=\id+a_1$ and $u_2=\id+a_2$ for square-zero $b$-skewselfadjoint endomorphisms $a_1$ and $a_2$.
Set $V_1:=\im a_1$ and $V_2:=\im a_2$.
Then $V_1 \subset \Ker a_1=V_1^{\bot_b}$ and likewise $V_2 \subset V_2^{\bot_b}$, i.e.\ $V_1$ and $V_2$ are totally $b$-singular, leading to
$\dim V_1 \leq \frac{1}{2} \dim V$ and  $\dim V_2 \leq \frac{1}{2} \dim V$.
Moreover, $V_1+V_2=V$ because $\im(u-\id) \subset V_1+V_2$ and $u-\id$ is invertible.
Hence $V=V_1 \oplus V_2$, $V_1=V_1^{\bot_b}$ and $V_2=V_2^{\bot_b}$. It follows that $a_1$ maps $V_2$ bijectively onto $V_1$,
and $a_2$ maps $V_1$ bijectively onto $V_2$.

Denote by $V_2^\circ:=\{\varphi \in V^*: \; \forall x \in V_2, \; \varphi(x)=0\}$ the dual-orthogonal of $V_2$.
As $V_2$ is totally $b$-singular, the isomorphism $L_b: x \mapsto b(-,x)$ induces an
injective linear mapping $V_2 \rightarrow (V_2)^\circ$, which turns out to be an isomorphism because $2\dim V_2=\dim V$.
Composing it with the natural (restriction) isomorphism from $(V_2)^\circ$ to $V_1^\star$,
we obtain an isomorphism $f : V_2 \rightarrow V_1^\star$ that takes every $x \in V_2$ to the linear form $y \in V_1 \mapsto b(y,x)$.
Then we consider the isomorphism
$$\Phi : \begin{cases}
V & \longrightarrow V_1 \times V_1^\star \\
x_1+x_2 & \longmapsto \bigl(x_1,\varepsilon\,f(x_2)\bigr) \quad \text{with $x_1 \in V_1$ and $x_2 \in V_2$.}
\end{cases}$$
Next, for all $x,y$ in $V$, we split $x=x_1+x_2$ and $y=y_1+y_2$ with $x_1,y_1$ in $V_1$ and $x_2,y_2$ in $V_2$,
and we see that
\begin{align*}
b(x,y) & =b(x_1,y_2)+b(x_2,y_1) \\
& =b(x_1,y_2)+\varepsilon \, b(y_1,x_2) \\
& = f(y_2)[x_1]+\varepsilon \,f(x_2)[y_1] \\
& =H_{V_1}^\varepsilon\bigl(\Phi(x),\Phi(y)\bigr).
\end{align*}
Hence, $\Phi$ is an isometry from $b$ to $H_{V_1}^\varepsilon$.
Now, setting $a'_1:=\Phi \circ a_1 \circ \Phi^{-1}$ and $a'_2 := \Phi \circ a_2 \circ \Phi^{-1}$, we
obtain two $H_{V_1}^{\varepsilon}$-skewselfadjoint endomorphisms.
As $a_1$ vanishes everywhere on $V_1$ and maps $V_2$ bijectively onto $V_1$,
we find that $a'_1$ vanishes everywhere on $V_1 \times \{0\}$ and maps $\{0\} \times V_1^\star$ bijectively onto $V_1 \times \{0\}$.
The associated isomorphism from $V_1^\star$ to $V_1$ then reads $L_B^{-1}$ for a unique non-degenerate bilinear form $B$ on $V_1$,
and hence $a'_1=w_B$.
Likewise, $a'_2$ vanishes everywhere on $\{0\} \times V_1^\star$ and maps $V_1 \times \{0\}$ bijectively onto $\{0\} \times V_1^\star$: the associated
linear map from $V_1$ to $V_1^\star$ reads $L_C$ for a unique non-degenerate bilinear form $C$ on $V_1$, and hence $a'_2=v_C$.
Since both $a'_1$ and $a'_2$ are $H_{V_1}^{\varepsilon}$-skewselfadjoint, we find that both $B$ and $C$ are skewsymmetric if $\varepsilon=1$,
and symmetric if $\varepsilon=-1$.

Noting that $\Phi \circ u \circ \Phi^{-1}=B \boxtimes C$, we conclude that
$B \boxtimes C$ belongs to the isometry group of $H_{V_1}^\varepsilon$ and that the
$\varepsilon$-isopairs $(b,u)$ and $(H_{V_1}^\varepsilon,B \boxtimes C)$ are isometric.
\end{proof}

Hence, in order to solve our problem for elements $u$ such that $u-\id$ is invertible, it essentially
remains to examine the possible invariants of boxed-products.

\subsection{Boxed-products as functions of pairs of forms}

Let $(b,c)$ be a pair of bilinear forms on a vector space $V$, and $(b',c')$ be a pair of bilinear forms on a vector space $V'$.
We say that $(b,c)$ is isometric to $(b',c')$ when there exists a vector space isomorphism
$\varphi : V \overset{\simeq}{\rightarrow} V'$ such that
$$\forall (x,y) \in V^2, \; b'\bigl(\varphi(x),\varphi(y)\bigr)=b(x,y) \quad \text{and} \quad
c'\bigl(\varphi(x),\varphi(y)\bigr)=c(x,y).$$
This defines an equivalence relation on the collection of pairs of bilinear forms with the same underlying vector space.

The orthogonal direct sum of two such pairs $(b_1,c_1)$ and $(b_2,c_2)$ is defined as
$(b_1,c_1) \bot (b_2,c_2) :=(b_1 \bot b_2,c_1 \bot c_2)$, and one checks that it is compatible with isometry
(if we replace a summand with an isometric one, then the sum is unchanged up to an isometry).

We shall say that a pair is \textbf{indecomposable} whenever it is non-trivial (i.e.\ defined on a non-zero vector space) and it is
not isometric to the orthogonal direct sum of two non-trivial pairs.

The following lemmas are proved in much the same way as the corresponding ones of section 2.4 of
\cite{dSPsquarezeroquadratic} (lemmas 2.6 and 2.7), so we leave the proofs
to the reader.

\begin{lemma}\label{cpequivalencelemma}
Let $\varepsilon \in \{-1,1\}$. Let $(b,c)$ and $(b',c')$ be pairs of bilinear forms on respective vector spaces $V$ and $V'$.
Assume that $b$ and $b'$ are non-degenerate and that
$(b,c) \simeq (b',c')$. Assume finally that $b,c,b',c'$ are all skewsymmetric if $\varepsilon=1$, and all symmetric if $\varepsilon=-1$.
Then
$$(H_V^{\varepsilon},b \boxtimes c) \simeq (H_{V'}^{\varepsilon},b' \boxtimes c').$$
\end{lemma}

\begin{lemma}\label{cpdirectsumlemma}
Let $\varepsilon \in \{-1,1\}$. Let $(b_1,c_1)$ and $(b_2,c_2)$ be pairs of bilinear forms on respective vectors spaces $V_1$ and $V_2$.
Assume finally that $b_1,c_1,b_2,c_2$ are all skewsymmetric if $\varepsilon=1$, and all symmetric if $\varepsilon=-1$.
Then
$$\bigl(H_{V_1 \times V_2}^{\varepsilon},(b_1\bot b_2) \boxtimes (c_1 \bot c_2)\bigr) \simeq
(H_{V_1}^\varepsilon,b_1 \boxtimes c_1) \bot (H_{V_2}^\varepsilon,b_2 \boxtimes c_2).$$
\end{lemma}

It is well known that the classification of pairs of symmetric (respectively, skewsymmetric) bilinear forms with the first form non-degenerate
comes entirely down to the one of pairs $(b,u)$ consisting of a non-degenerate symmetric (respectively, skewsymmetric) bilinear form $b$
on a vector space $V$ together with a $b$-selfadjoint endomorphism $u$ of $V$ (we call such a pair an $(\varepsilon,1)$-pair, where
$\varepsilon=1$ if $b$ is symmetric and $\varepsilon=-1$ if $b$ is skewsymmetric).

To see this, let $\varepsilon \in \{1,-1\}$, and let $(b,c)$ be a pair of bilinear forms on a vector space $V$, both symmetric if $\varepsilon=1$, and
both skewsymmetric if $\varepsilon=-1$. Then, $u:=L_b^{-1} L_c$ is an endomorphism of $V$, and one checks that
$(b,u)$ is an $(\varepsilon,1)$-pair. Conversely, given an $(\varepsilon,1)$-pair $(b,u)$, one sees that
$c : (x,y) \mapsto b(x,u(y))$ is a bilinear form that is symmetric if $\varepsilon=1$, skewsymmetric if $\varepsilon=-1$.
Moreover, one checks that, given pairs $(b,c)$ and $(b',c')$ of bilinear forms, every isometry from $(b,c)$ to $(b',c')$
turns out to be an isometry from $(b,L_b^{-1} L_c)$ to $(b',L_{b'}^{-1} L_{c'})$. Conversely, given
$(\varepsilon,1)$-pairs $(b,u)$ and $(b',u')$, every isometry from $(b,u)$ to $(b',u')$ turns out to be an isometry from
$(b,c)$ to $(b',c')$ where $c : (x,y) \mapsto b(x,u(y))$ and $c' : (x,y) \mapsto b'(x,u'(y))$.

\begin{lemma}\label{lemma:boxprodv}
Let $\varepsilon \in \{1,-1\}$.
Let $(b,c)$ be a pair of bilinear forms defined on the same vector space $V$, both symmetric if $\varepsilon=1$,
both skewsymmetric if $\varepsilon=-1$. Assume that $b$ is non-degenerate.
Set $v:=(b \boxtimes c)+(b \boxtimes c)^{-1}$ and $u:=L_b^{-1} L_c+2\id_V$.
Then
$$\forall (x,\varphi)\in V \times V^\star, \; v(x,\varphi)=(u(x),u^t(\varphi)).$$
\end{lemma}

\begin{proof}
One checks that $v=2\id+w_b v_c+v_c w_b$.
Let $x \in V$ and $\varphi \in V^\star$. Then
$$(w_bv_c+v_c w_b)(x,\varphi)=(L_b^{-1} L_c(x),L_c L_b^{-1}(\varphi)).$$
Denoting by $i : V \overset{\simeq}{\rightarrow} V^{\star\star}$
the standard double-duality isomorphism, we have $L_b=\varepsilon L_b^t \circ i$ and $L_c=\varepsilon L_c^t \circ i$, to the effect that
$L_c L_b^{-1}=L_c^t (L_b^t)^{-1}=(L_b^{-1} L_c)^t$.
Hence,
$$v(x,\varphi)=2\,(x,\varphi)+\bigl((u-2\id)(x),(u-2\id)^t(\varphi)\bigr)=\bigl(u(x),u^t(\varphi)\bigr).$$
\end{proof}

\subsection{The Wall invariants of boxed-products}\label{section:invariantsboxedproduct}

Before we compute the Wall invariants of boxed-products, we need to recall
the construction of the quadratic invariants of a pair $(b,c)$ of symmetric bilinear forms on the same vector space $V$, with
$b$ non-degenerate. Letting $P=(b,c)$ be such a pair, we set $u:=L_b^{-1} L_c$ and consider the corresponding
$(1,1)$-pair $(b,u)$. Let $p \in \Irr(\F) \cup \{t\}$ be of degree $d$, and $r \geq 1$ be an integer.
Consider the field $\L:=\F[t]/(p)$ equipped with the $\F$-linear form $e_p$ that takes the class of $1$ to $1$
and the class of $t^k$ to $0$ for all $k \in \lcro 1,d-1\rcro$. Consider the symmetric $\F$-bilinear form
$$\begin{cases}
\bigl(\Ker p(u)^{r}\bigr)^2 & \longrightarrow \F \\
(x,y) & \longmapsto b(x,p(u)^{r-1}[y]).
\end{cases}$$
Its radical is the intersection of $\Ker p(u)^r$
with the inverse image of $(\Ker p(u)^r)^{\bot_b}=\im p(u)^r$ under $p(u)^{r-1}$, and one easily
checks that this inverse image equals $\Ker p(u)^{r-1}+(\Ker p(u)^r \cap \im p(u))$.
Hence, the preceding bilinear form induces a non-degenerate $\F$-bilinear form $\overline{b_{p,r}}$ on the quotient space $V_{p,r}:=\Ker p(u)^r/(\Ker p(u)^{r-1}+(\Ker p(u)^r \cap \im p(u)))$, i.e.\ on the cokernel of the mapping from
$\Ker p(u)^{r+1}/\Ker p(u)^r$ to $\Ker p(u)^{r}/\Ker p(u)^{r-1}$ induced by $p(u)$.
Both these spaces are naturally seen as vector spaces over $\L$, and hence so is the said cokernel.
Since $u$ is $b$-selfadjoint we find that
$\overline{b_{p,r}}(x,\lambda y)=\overline{b_{p,r}}(\lambda x,y)$ for all
$\lambda \in \L$ and all $(x,y)\in (V_{p,r})^2$.
Just like in the construction of the Wall invariants, this yields a uniquely defined
non-degenerate symmetric bilinear form $P_{p,r}$ on $V_{p,r}$ such that
$$\forall (x,y,\lambda)\in (V_{p,r})^2 \times \L, \; \overline{b_{p,r}}(x,\lambda y)=e_p(\lambda P_{p,r}(x,y)).$$
The form $P_{p,r}$ is the \textbf{quadratic invariant} of $(b,u)$ with respect to $(p,r)$.

We will now compute the invariants of $(H_V^\varepsilon,b \boxtimes c)$
when $(b,c)$ is a pair of bilinear forms on the vector space $V$, both skewsymmetric if $\varepsilon=1$, both symmetric otherwise.
Using the compatibility of boxed-products with orthogonal direct sums and with isometry (see Lemmas \ref{cpequivalencelemma} and \ref{cpdirectsumlemma}), we will limit the computation to the
case where $(b,c)$ is indecomposable.
There are only two special cases to consider:
\begin{itemize}
\item If $b$ and $c$ are symmetric, $b$ is non-degenerate and $(b,c)$ is indecomposable, then $u:=L_b^{-1} L_c$ is cyclic and its minimal
polynomial equals $p^r$ for some $p \in \Irr(\F) \cup \{t\}$ and some $r \geq 1$. In that case, exactly one quadratic invariant of $(b,c)$
is non-trivial, namely $(b,c)_{p,r}$, and it is a non-degenerate symmetric bilinear form on a $1$-dimensional vector space over the field $\F[t]/(p)$.

\item If $b$ and $c$ are skewsymmetric, $b$ is non-degenerate and $(b,c)$ is indecomposable, then $u:=L_b^{-1} L_c$
is the direct sum of two cyclic endomorphisms with minimal polynomial $p^r$ for some $p \in \Irr(\F) \cup \{t\}$ and some $r \geq 1$.
\end{itemize}

\begin{Not}
Let $p$ be a monic polynomial of degree $d \geq 1$. Set
$$R(p):=t^d p(t+t^{-1}).$$
Note that $R(p)$ is a palindromial of degree $2d$.
\end{Not}

Let us recall corollary 4.7 from \cite{dSPregular}:

\begin{lemma}\label{lemma:qcorblockinvariants}
Let $N$ be an arbitrary matrix of $\Mat_n(\F)$, and $\delta$ be a nonzero scalar.
Denote by $r_1,\dots,r_a$ the invariant factors of $N$.
Then the invariant factors of
$$K(N):=\begin{bmatrix}
0_n & -I_n \\
I_n & N
\end{bmatrix}$$
are $R(r_1),\dots,R(r_a)$.
\end{lemma}

\begin{lemma}\label{lemma:invariantscrossedprod}
Let $b$ and $c$ be non-degenerate bilinear forms on a vector space $V$.
Denote by $p_1,\dots,p_r$ the invariant factors of $L_b^{-1} L_c+2\id$. Then
the invariant factors of $b \boxtimes c$ are $R(p_1),\dots,R(p_r)$.
\end{lemma}

\begin{proof}
Set $u:=L_b^{-1} L_c+2\id$.
Consider a basis $(e_1,\dots,e_n)$ of $V$. Then, $\mathbf{B}:=\bigl((0,L_b(e_1)),\dots,(0,L_b(e_n)),(e_1,0),\dots,(e_n,0)\bigr)$
is a basis of $V \times V^\star$.
The respective matrices of $v_c$ and $w_b$ in that basis equal
$$\begin{bmatrix}
0_n & A \\
0_n & 0_n
\end{bmatrix} \quad \text{and} \quad
\begin{bmatrix}
0_n & 0_n \\
I_n & 0_n
\end{bmatrix},$$
where $A$ stands for the matrix that represents $L_b^{-1} L_c$ in $(e_1,\dots,e_n)$.
It follows that the matrix of $b \boxtimes c$ in $\bfB$ equals
$$M:=\begin{bmatrix}
A+I_n & A \\
I_n & I_n
\end{bmatrix}.$$
Conjugating by the block-matrix $\begin{bmatrix}
I_n & I_n \\
0 & I_n
\end{bmatrix}$, we find that $M$ is similar to
$$M':=\begin{bmatrix}
A+2I_n & -I_n \\
I_n & 0_n
\end{bmatrix}.$$
By further conjugating with
$\begin{bmatrix}
0 & I_n \\
-I_n & 0
\end{bmatrix}$, we deduce that $M$ is similar to
$$M'':=\begin{bmatrix}
0_n & -I_n \\
I_n & A+2I_n
\end{bmatrix}.$$
As the invariant factors of $A+2I_n$ are the ones of $u$, we conclude by Lemma \ref{lemma:qcorblockinvariants} that the
invariant factors of $M$ are $R(p_1),\dots,R(p_r)$.
\end{proof}

Here is an immediate corollary:

\begin{lemma}\label{lemma:jordannumbersboxedprod}
Let $b$ and $c$ be symplectic forms on a vector space $V$.
Then the Jordan numbers of $b \boxtimes c$ are all even.
\end{lemma}

\begin{proof}
It is known (see e.g.\ \cite{Scharlaupairs}) that the invariant factors of $L_b^{-1}L_c$
come in pairs, i.e.\ their sequence is of the form $q_1,q_1,q_2,q_2,\dots,q_r,q_r$.
Hence, the ones of $L_b^{-1}L_c+2\id_V$ are $q_1(t-2),q_1(t-2),q_2(t-2),q_2(t-2),\dots,q_r(t-2),q_r(t-2)$,
and the ones of $b \boxtimes c$ are $R(q_1(t-2)),R(q_1(t-2)),R(q_2(t-2)),R(q_2(t-2)),\dots,R(q_r(t-2)),R(q_r(t-2))$.
The conclusion follows easily.
\end{proof}

A small lemma will be useful before we move forward:

\begin{lemma}\label{lemma:splitLaurent}
One has $\F[t,t^{-1}]=\F[t+t^{-1}]+t \F[t+t^{-1}]$.
\end{lemma}

\begin{proof}
Set $E:=\F[t+t^{-1}]+t \F[t+t^{-1}]$.
Firstly, $t^0=(t+t^{-1})^0$ belongs to $E$. Now, let $n \in \N$ be such that $\forall k \in \lcro -n,n\rcro, \; t^k \in E$.
Then $t(t+t^{-1})^n=t^{n+1}$ modulo $\Vect(t^k)_{-n \leq k \leq n}$, and hence modulo $E$, which shows that $t^{n+1} \in E$.
Finally, $(t+t^{-1})^{n+1}=t^{n+1}+t^{-(n+1)}$ modulo $\Vect(t^k)_{-n \leq k \leq n}$, and again we deduce that $t^{-(n+1)} \in E$.
Hence, by induction $t^k$ belongs to the linear subspace $E$ for all $k \in \Z$, and the conclusion follows.
\end{proof}

Next, we compute the Wall invariants of $(H_V^1,b \boxtimes c)$ when $b$ and $c$ are symplectic forms.

\begin{lemma}\label{lemma:hyperbolicboxedprod}
Let $(b,c)$ be an indecomposable pair of symplectic forms on a vector space $V$.
Then all the Wall invariants of $(H_V^1 ,b \boxtimes c)$
are hyperbolic.
\end{lemma}

\begin{proof}
Here, there exists an irreducible monic polynomial $m \neq t-2$ and an integer $r \geq 1$ such that
 $u:=L_b^{-1} L_c+2\id_V$ has exactly two primary invariants: $m^r$ and $m^r$.
 Lemma \ref{lemma:invariantscrossedprod} shows that the invariant factors of $b \boxtimes c$ are $R(m)^r$ and $R(m)^r$.

We note that $p:=R(m)$ is a palindromial of even degree, and as $m$ is irreducible either $p=qq^\sharp$ where $q$
is a monic irreducible polynomial such that $q \neq q^\sharp$, or $p$ is an even power of $t+1$.
In the last case, all the Wall invariants of $(H_V^1 ,b \boxtimes c)$ vanish.
In the second case, the primary invariants of $b \boxtimes c$
are $q^r$, $q^r$, $(q^\sharp)^r$ and $(q^\sharp)^r$, and again all the Wall invariants of $(H_V^1 ,b \boxtimes c)$ vanish.

Assume now that $p$ is irreducible.
Then $b \boxtimes c$ has exactly two primary invariants, namely $p^r$ and $p^r$, and
exactly one Hermitian invariant of $(H_V^1 ,b \boxtimes c)$ is non-trivial, namely $(H_V^1 ,b \boxtimes c)_{p,r}$:
it is a Hermitian form defined on a $2$-dimensional vector space over the field $\L:=\F[t,t^{-1}]/(p)$ equipped with the involution $\lambda \mapsto \lambda^\bullet$ that takes the class $\overline{t}$ to its inverse.

Set
$$v:=(b \boxtimes c)+(b \boxtimes c)^{-1}=2\id+w_bv_c+v_c w_b.$$
Remember from Lemma \ref{lemma:boxprodv} that
$$\forall (x,\varphi)\in V \times V^\star, \; v(x,\varphi)=(u(x),\varphi \circ u).$$
It follows that $v$ has exactly four primary invariants, all equal to $m^r$.
Consequently, $\Ker m(v)^{r-1}=\im m(v)$, and hence the Wall invariant $(H_V^1 ,b \boxtimes c)_{p,r}$ is defined on the quotient space
$(V \times V^\star)/\im m(v)$. Besides $\im m(v)=\im m(u) \times \im m(u)^t$.

Let us choose $x \in V \setminus \im m(u)$, set $X:=(x,0)$ and denote by $\overline{X}$ the class of $X$ modulo $\im m(v)$.
Note that $\overline{X} \neq 0$. We shall prove that $X$ is isotropic for $(H_V^1 ,b \boxtimes c)_{p,r}$, which will complete the proof.

Next, we equip $V \times V^\star$ by the structure of $\F[t,t^{-1}]$-module induced by $b \boxtimes c$, and
$V$ with the structure of $\F[t]$-module induced by $u$.

Noting that $v(x,0)=(u(x),0)$, we obtain
$\forall q \in \F[t], \; q(t+t^{-1})\,(x,0)=(q\,x,0)$ and it follows that
\begin{equation}\label{eq:even}
\forall q \in \F[t], \; H_V^1\bigl((x,0),q(t+t^{-1})\,(x,0)\bigr)=0.
\end{equation}
Let $q \in \F[t]$. Then, $tq(t+t^{-1})\,(x,0)=\bigl(q\,x,L_c(q\,x)\bigr)$.
Hence
\begin{align*}
H_V^1\bigl((x,0),tq(t+t^{-1}).(x,0)\bigr) & =L_c(q(u)[x])[x] \\
& =c(x,q(u)[x]) \\
& =b(x,(u-2\id)q(u)[x]).
\end{align*}
Yet $b$ and $c$ are skewsymmetric and $u$ is $b$-selfadjoint, and hence $b(x,u^k(x))=0$ for every integer
$k \geq 0$: indeed, if $k=2l$ for some integer $l$, one writes $b(x,u^k(x))=b(u^l(x),u^l(x))=0$; if $k=2l+1$ for some integer $l$, one writes $$b(x,u^k(x))=b(u^{l}(x),u^{l+1}(x))=c(u^l(x),u^{l}(x))+2b(u^{l}(x),u^{l}(x))=0.$$
Hence,
\begin{equation}\label{eq:odd}
H_V^1\bigl((x,0),tq(t+t^{-1}).(x,0)\bigr)=b\bigl(x,((t-2)q)(u)[x]\bigr)=0.
\end{equation}
With the help of Lemma \ref{lemma:splitLaurent},
combining \eqref{eq:even} and \eqref{eq:odd} yields $H_V^1(X,\lambda\,X)=0$ for all $\lambda \in \F[t,t^{-1}]$.

It follows that, in the quotient $\F[t,t^{-1}]$-module $V/\im m(v)$,
the class $\overline{X}$ generates a submodule that is orthogonal to itself
for the bilinear form induced by $(Y,Z) \mapsto H_V^1(Y,m^{r-1}\,Z)$.

By coming back to the definition of the Hermitian Wall invariant of order $r$ with respect to $p$,
we conclude that $(H_V^1,b \boxtimes c)_{p,r}(\overline{X},\overline{X})=0$, which completes the proof.
\end{proof}

Finally, we compute the Wall invariants of $(H_V^{-1},b \boxtimes c)$ when $b$ and $c$ are symmetric forms.

\begin{Def}
A symmetric bilinear form $B$ on a vector space $V$ is said to \textbf{represent} a non-zero scalar $\alpha$
if there exists a vector $x$ of $V$ such that $B(x,x)=\alpha$.
\end{Def}

\begin{lemma}\label{lemma:skewrepresentabledim1}
Let $(b,c)$ be a pair of non-degenerate symmetric bilinear forms on a vector space $V$.
Assume that $u:=2\id_V+L_b^{-1}L_c$ is cyclic with minimal polynomial $m^r$ for some monic irreducible polynomial $m$
and some $r \geq 1$. Assume furthermore that $p:=R(m)$ is irreducible.
Assume finally that the quadratic invariant $(b,u)_{m,r}$ represents, for some polynomial $\alpha(t)\in \F[t]$ (not divisible by $m$),
the class of $\alpha(t)$ modulo $m$. \\
Then $(H_V^{-1},b \boxtimes c)$ has exactly one nontrivial Wall invariant, namely $(H_V^{-1},b \boxtimes c)_{p,r}$, which is defined on a
$1$-dimensional vector space, and it represents the element $\beta:=(\overline{1-t})(\overline{1+t})^{-1}\,\overline{\alpha(t+t^{-1})}$ of the quotient field $\F[t,t^{-1}]/(p)$.
\end{lemma}

\begin{proof}
This time around, Lemma \ref{lemma:qcorblockinvariants} shows that the sole invariant factor of $b \boxtimes c$ is $p^r$.
Here $p$ is an irreducible palindromial of even degree.

We consider the fields $\mathbb{M}:=\F[t]/(m)$ and $\mathbb{L}:=\F[t,t^{-1}]/(p)$.
For a Laurent polynomial $a \in \F[t,t^{-1}]$, we denote by $\overline{a(t)}^{\mathbb{L}}$ its class modulo $p$, and
for a standard polynomial $a \in \F[t]$, we denote by $\overline{a(t)}^{\mathbb{M}}$ its class modulo $m$.

Set $v:=(b \boxtimes c)+(b \boxtimes c)^{-1}$. By Lemma \ref{lemma:boxprodv}, we have
$$\forall (x,\varphi) \in V \times V^\star, \; v(x,\varphi)=(u(x),\varphi \circ u).$$
Just like in the proof of Lemma \ref{lemma:hyperbolicboxedprod}, the form $\overline{B}$ that is used to construct the Wall invariant
$(H_V^{-1},b \boxtimes c)_{p,r}$ is defined on the quotient space $(V \times V^\star)/(\im m(u) \times \im m(u)^t)$,
whereas the one that is used to construct the quadratic invariant $(b,u)_{m,r}$ is defined
on the quotient space $V /\im m(u)$. The assumption on $m$ shows that we can choose
a vector $x$ of $V \setminus \im m(u)$ such that
$$\forall q \in \F[t], \; b\bigl(x,(qm^{r-1})(u)[x]\bigr)=e_m\bigl(\overline{\alpha(t)q(t)}^{\mathbb{M}}\bigr).$$
To simplify the notation, we now endow $V \times V^\star$ with the structure of $\F[t,t^{-1}]$-module induced by $b \boxtimes c$,
whereas $V$ is endowed with the structure of $\F[t]$-module induced by $u$.
Let us consider the vector $X:=(x,0) \in V \times V^\star$.
Note that $p^{r-1}\,X=(m^{r-1}\,x,0)$.

Let $s \in \F[t]$.
Noting that $\beta$ is skew-Hermitian in $\F[t,t^{-1}]/(p)$,
we find, like in the proof of Lemma \ref{lemma:hyperbolicboxedprod}, that
$$f_p(\beta \overline{s(t+t^{-1})})=0=H_V^{-1}\bigl((x,0),(sm^{r-1})(t+t^{-1})\,(x,0)\bigr).$$
Now, set
$$q:=(t-t^{-1})\,s(t+t^{-1})=2t s(t+t^{-1})-(ts)(t+t^{-1}).$$
Then
\begin{align*}
H_V^{-1}\bigl((x,0),m^{r-1}(t+t^{-1})q\,(x,0)\bigr)
& =2 H_V^{-1}\bigl((x,0),(sm^{r-1} x,L_c(s m^{r-1}\, x))\bigr) \\
& =-2\, L_c(s m^{r-1}\, x)[x] \\
& =-2\, c(x,s m^{r-1}x) \\
& =-2\,b(x,(t-2)s m^{r-1}x) \\
& = -2\, e_m\bigl(\overline{(t-2)s\alpha}^{\mathbb{M}}\bigr) \\
& = - f_p\Bigl(\overline{(t+t^{-1}-2)\, s(t+t^{-1})\, \alpha(t+t^{-1})}^{\mathbb{L}}\Bigr) \\
& = f_p\Bigl(\beta\,\overline{q(t)}^{\mathbb{L}}\Bigr),
\end{align*}
where we have used the identity
$$(1-t)(1+t)^{-1} (t-t^{-1})=t^{-1} (1-t)(1+t)^{-1} (t+1)(t-1)=-t^{-1}(t-1)^2=-(t+t^{-1}-2).$$
Using Lemma \ref{lemma:splitLaurent}, we deduce by linearity that
$$\forall q \in \F[t,t^{-1}], \;
H_V^{-1}\bigl((x,0),(m(t+t^{-1}))^{r-1}) q\,(x,0)\bigr)=f_p\bigl(\beta\,\overline{q}^\L\bigr),$$
and we conclude that $\beta=(H_V^{-1},b \boxtimes c)_{p,r}(\overline{X},\overline{X})$,
where $\overline{X}$ stands for the class of $X$ in $(V \times V^\star)/\im p(b \boxtimes c)$.
Hence, the claimed result is proved.
\end{proof}

\section{Symplectic groups}\label{section:symplectic}

\subsection{Decompositions into two factors}

Here, we prove Theorem \ref{theo:symplectic}.
So, we take a $-1$-isopair $(b,u)$ such that $u$ has no Jordan cell of odd size for the eigenvalue $-1$.
We will prove that $u$ is the product of two $U_2$-elements of $\Sp(b)$.

To do so, we can limit our study to the case where $(b,u)$ is indecomposable as a $-1$-isopair.
Then, there are four possibilities:
\begin{itemize}
\item Case 1. The invariant factors of $u$ read $(t-1)^r,(t-1)^r$ for some odd integer $r \geq 1$.
\item Case 2. The primary invariants of $u$ read $p^r,(p^{\sharp})^r$ for some $p \in \Irr(\F)$ such that $p \neq p^\sharp$ and some
integer $r \geq 1$.
\item Case 3. $u$ is cyclic with minimal polynomial $(t-\eta)^r$ for some even integer $r \geq 1$ and some $\eta \in \{1,-1\}$.
\item Case 4. $u$ is cyclic with minimal polynomial $p^r$ for some palindromial $p \in \Irr(\F) \setminus \{t\pm 1\}$ and some integer $r \geq 1$.
\end{itemize}

\noindent \textbf{Case 1. The invariant factors of $u$ equal $(t-1)^r$, $(t-1)^r$ for some odd integer $r \geq 1$.} \\
Let us choose a cyclic endomorphism $v$ of a vector space $W$, with minimal polynomial $(t-1)^r$.
Then it is known that $h(v)$ has exactly two invariants factors, namely $(t-1)^r$ and $(t-1)^r$.
By the classification of conjugacy classes in symplectic groups, it follows that $(b,u)$ is isometric to $H_{-1}(v)$.
Since $v$ is the product of two $U_2$-elements of $\GL(W)$, it follows from Proposition \ref{prop:hyperbolicexpansionU2split} that $H_{-1}(v)$
is $U_2$-splittable, and hence so is $(b,u)$.

\vskip 2mm
\noindent \textbf{Case 2. The primary invariants of $u$ equal $p^r,(p^{\sharp})^r$ for some $p \in \Irr(\F)$ such that $p \neq p^\sharp$ and some
integer $r \geq 1$.}

Let us write $pp^\sharp=R(m)$ where $m \in \F[t]$. By the classification of pairs of symmetric bilinear forms, we can find a pair $(B,C)$
of symmetric bilinear forms on a vector space $W$ such that $2\id_W+L_B^{-1} L_C$ is cyclic with minimal polynomial $m^r$.
By Lemma \ref{lemma:invariantscrossedprod}, the endomorphism $B \boxtimes C$ is cyclic with minimal polynomial $(pp^\sharp)^r$.
By the classification of $-1$-isopairs, it follows that $(H_W^{-1},B \boxtimes C) \simeq (b,u)$.
Since $(H_W^{-1},B \boxtimes C)$ is $U_2$-splittable, we conclude that so is $(b,u)$.

\vskip 2mm
\noindent \textbf{Case 3. $u$ is cyclic with minimal polynomial $(t-\eta)^r$ for some $\eta \in \{1,-1\}$ and some even integer $r>0$.} \\
By the classification of pairs of symmetric bilinear forms, we can find a pair $(B,C)$
of symmetric bilinear forms on a vector space $W$ such that $2\id_W+L_B^{-1} L_C$ is cyclic with minimal polynomial $(t-2\eta)^{r/2}$.
By Lemma \ref{lemma:invariantscrossedprod}, the endomorphism $B \boxtimes C$ is cyclic with minimal polynomial $(t-\eta)^r$.
Let us choose a nonzero value $\alpha$ represented by the nonzero Wall invariant $(b,u)_{t-\eta,r}$, and
a nonzero value $\beta$ represented by the nonzero Wall invariant $(H_W^{-1},B \boxtimes C)_{t-\eta,r}$.
Then, for $\lambda:=\alpha \beta^{-1}$ we have that the $1$-dimensional Wall invariants $(\lambda H_W^{-1},B \boxtimes C)_{t-\eta,r}$
and $(b,u)_{t-\eta,r}$ are equivalent. It follows that $(b,u) \simeq (\lambda H_W^{-1},B \boxtimes C)$.
Since $\lambda H_W^{-1}$ and $H_W^{-1}$ have the same symplectic group, we obtain that
$(\lambda H_W^{-1},B \boxtimes C)$ is $U_2$-splittable because $(H_W^{-1},B \boxtimes C)$ is $U_2$-splittable.
Hence, $(b,u)$ is $U_2$-splittable.

\vskip 2mm
\noindent \textbf{Case 4. $u$ is cyclic with minimal polynomial $p^r$ for some irreducible palindromial $p \in \Irr(\F) \setminus \{t\pm 1\}$ and some $r \geq 1$.}

We can write $p=R(m)$ for some irreducible polynomial $m \neq t \pm 2$.
The pair $(b,u)$ has a sole non-zero Wall invariant, namely $(b,u)_{p,r}$.
This invariant is a skew-Hermitian form on a $1$-dimensional vector space over $\L:=\F[t,t^{-1}]/(p)$.
Let us take a nonzero value $\beta$ represented by this form.
In $\L$, the element $(\overline{1-t})(\overline{1+t})^{-1}$ is non-zero and skew-Hermitian, so
$\beta=(\overline{1-t})(\overline{1+t})^{-1}\overline{\alpha(t+t^{-1})}$ for some polynomial $\alpha \in \F[t]$ that is not divided by $m$.
It follows from the classification of pairs of symmetric bilinear forms that we can find a vector space $W$ together with a pair
$(B,C)$ of symmetric bilinear forms on $W$ such that $U:=2\id_W+L_B^{-1} L_C$ is cyclic with minimal polynomial $m^r$
and the class of $\alpha$ in $\F[t]/(m)$ is represented by the quadratic invariant $(B,U)_{m,r}$.
By Lemma \ref{lemma:skewrepresentabledim1}, we conclude that $\beta$ is represented by $(H_W^{-1}, B \boxtimes C)_{p,r}$.

Hence, by the classification of $1$-isopairs we conclude that $(b,u) \simeq (H_V^{-1}, B \boxtimes C)$.
And since $(H_V^{-1}, B \boxtimes C)$ is $U_2$-splittable we conclude that so is $(b,u)$.

This completes the proof of Theorem \ref{theo:symplectic}.

\subsection{Decompositions into three factors}

Here, we prove that every element of a symplectic group is the product of three $U_2$-elements (Theorem \ref{theo:symplectic3}).
It will suffice to consider the situation of an indecomposable $-1$-isopair $(b,u)$.
By Theorem \ref{theo:symplectic}, $u$ is the product of two, and hence of three $U_2$-elements of $\Sp(b)$
(simply insert the identity) unless $u+\id$ is nilpotent with exactly two Jordan cells, both of odd size $r \geq 1$.

Assume now that $u+\id$ is nilpotent with exactly two Jordan cells, both of odd size $r \geq 1$.
Assume first that $r=1$. Then $u=-\id$, the underlying vector space $V$ has dimension $2$
and $\Sp(b)=\SL(V)$.
Choosing a unipotent $u_1 \in \SL(V) \setminus \{\id_V\}$, we see that
$-u_1^{-1}$ has a sole Jordan cell, which is of size $2$ and associated with the eigenvalue $-1$.
It follows from Theorem \ref{theo:symplectic} that $-u_1^{-1}=u_2u_3$ for $U_2$-elements $u_2,u_3$ of $\Sp(b)$, and hence
$-\id_V=u_1u_2u_3$ is the product of three $U_2$-elements.

Assume finally that $r>2$.
We have $(b,u) \simeq H_1(v)$, where $v$ is a cyclic automorphism of a vector space $V$, with minimal polynomial $(t+1)^r$.

We recall the following lemma (see e.g. proposition 3.5 in \cite{dSPinvol3}):

\begin{lemma}\label{lemma:uadapted}
Let $v$ be a cyclic automorphism of a vector space $V$ of dimension $n$,
and let $p \in \F[t]$ be a monic polynomial of degree $n$ such that $p(0)=(-1)^n \det v$.
Then there exists a $U_2$-automorphism $u_1$ of $V$ such that $u_1^{-1} v$
is cyclic with minimal polynomial $p$.
\end{lemma}

Since $r \geq 2$, we can find a monic polynomial $p \in \F[t]$ such that $p(0)=1$
and $p(-1) \neq 0$ (e.g.\ $p(t)=t^r+(-1)^r t+1$).
By Lemma \ref{lemma:uadapted}, we can choose a $U_2$-automorphism $v$ of $V$ such that $u_1^{-1} v$ is cyclic with minimal polynomial $p$.
Hence $h(u_1^{-1} v)$ belongs to $\Sp(H_V^{-1})$ but $-1$ is outside its spectrum. By Theorem \ref{theo:symplectic},
$h(u_1^{-1} v)$ is the product of two $U_2$-elements of $\Sp(H_V^{-1})$. Hence $h(u)=h(u_1) h(u_1^{-1} v)$
is the product of three  $U_2$-elements of $\Sp(H_V^{-1})$. We conclude that $u$ is the product of three $U_2$-elements of $\Sp(b)$.
This completes the proof of Theorem \ref{theo:symplectic3}.

\section{Orthogonal groups: automorphisms whose spectrum does not contain $1$}\label{section:orthogonalnonunipotent}

In this short section, we characterize, among the $1$-isopairs $(b,u)$ in which $u-\id$ is invertible, the $U_2$-splittable ones.
Here, Theorem \ref{theo:orthogonal1noeigenvalue} is enriched as follows:

\begin{theo}\label{theo:orthogonal1noeigenvalueenriched}
Let $(b,u)$ be a $1$-isopair such that $u-\id$ is invertible.
The following conditions are equivalent:
\begin{enumerate}[(i)]
\item $u$ is the product of two $U_2$-elements of $\Ortho(b)$.
\item All the Jordan numbers of $u$ are even, all the Wall invariants of $(b,u)$ are hyperbolic, and
$u$ has no Jordan cell of odd size for the eigenvalue $-1$.
\item $(b,u) \simeq H_1(v)$ for some automorphism $v$ of a vector space
such that $v$ is similar to its inverse and has no Jordan cell of odd size for the eigenvalue $-1$.
\item $(b,u) \simeq H_1(v)$ for some automorphism $v$ of a vector space $V$ such that $v$
is the product of two $U_2$-automorphisms of $V$.
\end{enumerate}
\end{theo}

First of all, let us assume that condition (i) holds. Since $1$ is not an eigenvalue of $u$,
we know from Proposition \ref{prop:caracboxedprod} that $(b,u)$ is isometric to the boxed-product $(H_W^1,B \boxtimes C)$
for some pair $(B,C)$ of symplectic forms on a vector space $W$. By Lemmas \ref{lemma:jordannumbersboxedprod} and \ref{lemma:hyperbolicboxedprod},
all the Jordan numbers of $u$ are even and all the Wall invariants of $(b,u)$ are hyperbolic.
Finally, since $u$ is the product of two $U_2$-automorphisms, it is known by the Wang-Wu-Botha theorem (Theorem \ref{theo:Botha}) that
$u$ has no Jordan cell of odd size for the eigenvalue $-1$. Hence, condition (ii) holds.

Next, the Wang-Wu-Botha theorem also shows that condition (iii) implies condition (iv).
And by Proposition \ref{prop:hyperbolicexpansionU2split}, condition (iv) implies condition (i).

Assume finally that condition (ii) holds. Let us prove that condition (iii) holds. Since all the Jordan numbers of $u$ are even, we can find a vector space
$V$ and an automorphism $v$ of $V$ such that $n_{p,r}(v)=\frac{1}{2}\,n_{p,r}(u)$ for all $p \in \Irr(\F)$ and all $r \geq 1$.
Note that $v$ still has no Jordan cell of odd size for the eigenvalue $-1$. Moreover, since $u$ is similar to its inverse we have
$$\forall p \in \Irr(\F), \; \forall r \geq 1, \; n_{p,r}(v^{-1})=n_{p^\sharp,r}(v)=
\frac{1}{2}\,n_{p^\sharp,r}(u)=\frac{1}{2}\,n_{p,r}(u)=n_{p,r}(v).$$
Hence $v$ is similar to its inverse. We claim finally that $(b,u) \simeq H_1(v)$.
To see this, we use the classification of $1$-isopairs. Since
$n_{p,r}(h(v))=n_{p,r}(v)+n_{p^\sharp,r}(v)=n_{p,r}(u)$ for all $p \in \Irr(\F)$ and all $r \geq 1$,
it will suffice to see that $(b,u)$ and $H_1(v)$ have the same Wall invariants.
So, let $p \in \Irr(\F) \setminus \{t\pm 1\}$ be an irreducible palindromial, and let $r \geq 1$.
The Wall invariants $(b,u)_{p,r}$ and $H_1(v)_{p,r}$ are hyperbolic Hermitian forms defined on spaces of the same
dimension over the field $\F[t,t^{-1}]/(p)$, and hence they are equivalent.
Moreover, $1$ is neither an eigenvalue of $u$ nor one of $h(v)$, and
there is no Jordan cell of odd size associated with the eigenvalue $-1$ for $u$ nor for $h(v)$; there is no other nontrivial Wall invariant
to consider. We conclude that $(b,u) \simeq H_1(v)$.

Hence, Theorem \ref{theo:orthogonal1noeigenvalueenriched} is proved.

\section{Orthogonal groups: unipotent automorphisms}\label{section:unipotent}

This section is devoted to the proof of Theorem \ref{theo:orthogonalunipotent}.
In Section \ref{section:unipotent1}, for a unipotent $1$-isopair $(b,u)$ we compute the Witt type of $b$
as a function of the Wall invariants of $(b,u)$.
In Section \ref{section:unipotenttwisted}, we construct a specific $U_2$-splittable unipotent
$1$-isopair $(b,u)$ where $u$ has exactly two Jordan cells, one of size $2k+1$ and one of size $2k-1$, and
we compare its Wall invariants associated with $(t-1,2k+1)$ and $(t-1,2k-1)$.
The main result of Section \ref{section:unipotent1} is then used to give more general constructions of $1$-isopairs
of this type. Then, in Section \ref{section:reconstruction}, these pairs are combined with hyperbolic extensions of unipotent automorphisms
to obtain the implication (ii) $\Rightarrow$ (i) in Theorem \ref{theo:orthogonalunipotent}.

The remainder of the section is devoted to the proof of the implication (i) $\Rightarrow$ (ii) in Theorem \ref{theo:orthogonalunipotent}.
The first part of the proof consists in reducing the study to the case where $(u-\id)^3=0$, using an induction process and well-chosen induced pairs:
this is done in Sections \ref{section:induced} and \ref{section:induction}. Finally, the case where $(u-\id)^3=0$ is dealt with in Section \ref{section:keylemmaproof}.

The reader acquainted with \cite{dSPsquarezeroquadratic} will recognize a similar strategy, with several subtle differences (in particular, in the proof of the case where $(u-\id)^3=0$).

\subsection{On quadratic Wall invariants and the type of $b$}\label{section:unipotent1}

Let $(b,u)$ be a unipotent $1$-isopair. Here, we discuss the relationship between the Witt type of $b$
and the Wall invariants of $(b,u)$.

We start from the case where all the Wall invariants of $(b,u)$ but one vanish.
So, assume that, for some $k \geq 1$,  $u$ has only Jordan cells of size $k$.

Note first that
$$\Ker(u-\id)^i=\bigl(\im ((u-\id)^i)^\star\bigr)^{\bot_b}=\bigl(\im (u^{-1}-\id)^i\bigr)^{\bot_b}=\bigl(\im (u-\id)^i\bigr)^{\bot_b}.$$
\noindent \textbf{Case 1:} $k=2l$ for some $l \geq 1$. Then $\Ker (u-\id)^i=\im (u-\id)^{k-i}$ for all $i \in \lcro 0,k\rcro$ (this is obvious on a Jordan cell of size $k$). In particular $\Ker (u-\id)^l=\im (u-\id)^l=(\Ker (u-\id)^l)^{\bot_b}$, and it follows that $b$ is hyperbolic.

\noindent \textbf{Case 2:} $k=2l+1$ for some $l \geq 0$. This time around, $(\Ker (u-\id)^l)^{\bot_b}=\im (u-\id)^l=\Ker (u-\id)^{l+1}$.
Classically, this yields that $b$ is Witt-equivalent to the non-degenerate bilinear form it induces on the quotient space
$\Ker (u-\id)^{l+1}/\Ker (u-\id)^l$. Set $v:=u+u^{-1}$. The bilinear form $B : (x,y)\in V^2 \mapsto b(x,(v-2\id)^l(y))$ reads
$(x,y) \in V^2 \mapsto (-1)^l b((u-\id)^l(x),(u-\id)^l(y))$.
Hence, the form induced by $B$ on $V/\im (u-\id)=\Ker (u-\id)^{2l+1}/\Ker (u-\id)^{2l}$
is equivalent to the form induced by $(-1)^l b$ on the quotient space $\Ker (u-\id)^{l+1}/\Ker (u-\id)^{l}$ through the isomorphism
$\Ker (u-\id)^{2l+1}/\Ker (u-\id)^{2l} \overset{\simeq}{\longrightarrow} \Ker (u-\id)^{l+1}/\Ker (u-\id)^{l}$ induced by $(u-\id)^l$.
Hence, $b$ is Witt-equivalent to $(-1)^l (b,u)_{t-1,2l+1}$.

\vskip 3mm
Let us come back to the general case of a unipotent $1$-isopair $(b,u)$. The classification of 1-isopairs shows that for every integer $k \geq 1$,
there exists a unipotent $1$-isopair $P^{(k)}$ for which:
\begin{itemize}
\item If $k$ is odd then the Wall invariant $(P^{(k)})_{t-1,i}$ vanishes
for every odd positive integer $i \neq k$, the Wall invariant $(P^{(k)})_{t-1,k}$ is equivalent to $(b,u)_{t-1,k}$, and
$P^{(k)}$ has no Jordan cell of even size;
\item If $k$ is even then $P^{(k)}$ has only Jordan cells of size $k$, and it has as many such cells as $(b,u)$.
\end{itemize}
It follows that $(b,u) \simeq \underset{k \geq 1}{\bot} P^{(k)}$.
Using the previous special cases, we conclude:

\begin{prop}\label{prop:wittequivalencetype}
Let $(b,u)$ be a unipotent $1$-isopair.
Then $b$ is Witt-equivalent to $\underset{l \geq 0}{\bot} (-1)^l (b,u)_{t-1,2l+1}$.
\end{prop}

\subsection{The splitting of twisted pairs}\label{section:unipotenttwisted}

\begin{lemma}\label{lemma:twisted}
Let $k$ be a positive integer. Let $V$ be a $4k$-dimensional vector space
together with a hyperbolic bilinear form $b : V^2 \rightarrow \F$.
Then there exists $u \in \Ortho(b)$ such that $(b,u)$ is $U_2$-splittable and
$u$ is unipotent with exactly two Jordan cells, one of size $2k+1$ and one of size $2k-1$.
\end{lemma}

\begin{proof}
We choose a $b$-hyperbolic basis $\bfB=(e_1,\dots,e_{2k},f_1,\dots,f_{2k})$ of $V$.
The matrix of $b$ in that basis is
$S:=\begin{bmatrix}
0 & I_{2k} \\
I_{2k} & 0
\end{bmatrix}$. We convene that $e_l=0$ for every integer $l \leq 0$.

Next, we define $a_1 \in \End(V)$ by $a_1(e_{2i})=e_{2i-1}$ and
$a_1(f_{2i-1})=- f_{2i}$ for all $i \in \lcro 1,k\rcro$, and $a_1$ maps all the other vectors of $\bfB$ to $0$.
We define $a_2 \in \End(V)$ by $a_2(e_{2i+1})=e_{2i}$ and
$a_2(f_{2i})=- f_{2i+1}$ for all $i \in \lcro 1,k-1\rcro$,
$a_2(f_1)=e_{2k}$ and $a_2(f_{2k})=- e_1$, and $a_2$ maps all the other vectors of $\bfB$ to $0$.
It is easily checked on the vectors of $\bfB$ that $a_1^2=0$ and $a_2^2=0$.

The respective matrices $M_1$ and $M_2$ of $a_1$ and $a_2$ in $\bfB$ are of the form
$$M_1=\begin{bmatrix}
A_1 & 0 \\
0 & - A_1^T
\end{bmatrix} \quad \text{and} \quad
\begin{bmatrix}
A_2 & C \\
0 & - A_2^T
\end{bmatrix}$$
where $C^T=-C$.
One checks that both $SM_1$ and $SM_2$ are skewsymmetric. Hence $a_1$ and $a_2$ are $b$-skewselfadjoint, and we conclude that
$\id_V+a_1$ and $\id_V+a_2$ are $U_2$-elements in $\Ortho(b)$.

Let us show that $u:=(\id_V+a_1)(\id_V+a_2)=\id_V+a_1+a_2+a_1a_2$ has the required property.
It will suffice to see that $u$ is unipotent with exactly two Jordan cells, one of size $2k+1$ and one of size $2k-1$.
Let us set
$$v:=u+u^{-1}-2\id=a_1a_2+a_2a_1.$$
For all $l \geq 0$, note that $v^l= u^{-l} (u-\id)^{2l}$ and hence
$\rk (u-\id)^{2l}=\rk v^l$ and $\rk (u-\id)^{2l+1}=\rk((u-\id) v^l)$.

Next, we have $v(e_i)=e_{i-2}$ for all $i \in \lcro 1,2k\rcro$, and
$v(f_1)=f_3+e_{2k-1}$ and $v(f_i)=f_{i+2}$ for all $i \in \lcro 2,2k-2\rcro$, $v(f_{2k-1})=e_1$ and $v(f_{2k})=0$.
It easily follows that $v^k$ vanishes at every vector of $\mathbf{B}$ with the exception of
$f_1$, for which $v^k(f_1)=2\,e_1$. And then $v^{k+1}=0$, and better still $(u-\id)\,v^k=0$. As $v=u^{-1}(u-\id)^2$, it follows that
$u$ is unipotent, $\rk((u-\id)^{2k})=\rk v^k=1$ and $(u-\id)^{2k+1}=0$.

Now, remember that, for all $i \geq 0$, the difference $\rk((u-\id)^i)-\rk((u-\id)^{i+1})$ is the number of Jordan cells of
$u$ of size greater than $i$ (for the eigenvalue $1$).
In particular, we find that $u$ has exactly one Jordan cell of size at least $2k+1$ for $1$, and its size equals $2k+1$.

Finally, one computes on the basis $\bfB$ that $\im v^{k-1}=\Vect(e_1,e_2,e_3+f_{2k-1},f_{2k})$,
from which we find that $(u-\id)(\im v^{k-1})=\Vect(e_1,-f_{2k}+e_2+e_1)$ and hence $\rk (u-\id)^{2k-1}=2$. Therefore
$\rk(u-\id)^{2k-2}-\rk (u-\id)^{2k-1}=2$. Hence, $u$ has exactly two Jordan cells of size at least $2k-1$.
Since $u$ is unipotent with one Jordan cell of size $2k+1$, and since $\dim V=(2k+1)+(2k-1)$, we conclude that $u$ has exactly two Jordan cells, one of size $2k+1$
and one of size $2k-1$.
\end{proof}

\begin{cor}\label{cor:twistedcor}
Let $(b,u)$ be a unipotent $1$-isopair with exactly two Jordan cells, of respective sizes $2k+1$ and
$2k-1$. Assume that $(b,u)_{t-1,2k+1} \simeq (b,u)_{t-1,2k-1}$. Then
$(b,u)$ is $U_2$-splittable.
\end{cor}

\begin{proof}
By the previous lemma, there is a $U_2$-splittable unipotent $1$-isopair $(b',u')$ with exactly two Jordan cells, of respective sizes $2k+1$ and $2k-1$, and with $b'$ hyperbolic.
By Proposition \ref{prop:wittequivalencetype}, $b'$ is Witt-equivalent to $(-1)^k (b',u')_{t-1,2k+1} \bot (-1)^{k-1}  (b',u')_{t-1,2k-1}$,
and it follows that $(-1)^k (b',u')_{t-1,2k+1} \simeq -(-1)^{k-1}  (b',u')_{t-1,2k-1}$,
that is $(b',u')_{t-1,2k+1} \simeq (b',u')_{t-1,2k-1}$.
Now, choose a nonzero value $\alpha$ (respectively, $\beta$) represented by the Wall invariant $(b,u)_{t-1,2k-1}$ (respectively, by $(b',u')_{t-1,2k-1}$).
Then, for $b'':=\alpha \beta^{-1}b'$, we see that $(b'',u')$ is a $1$-isopair, with sole non-trivial Wall invariants $(b'',u')_{t-1,2k-1}= \alpha \beta^{-1} (b',u')_{t-1,2k-1}$ and
$(b'',u')_{t-1,2k+1}= \alpha \beta^{-1} (b',u')_{t-1,2k+1}$. Both represent the scalar $\alpha$.
It follows from our assumptions that $(b'',u')_{t-1,2k-1} \simeq (b,u)_{t-1,2k-1}$ and $(b'',u')_{t-1,2k+1} \simeq (b,u)_{t-1,2k+1}$,
and we deduce from the classification of $1$-isopairs that $(b'',u')$ is isometric to $(b,u)$.
As $\Ortho(b')=\Ortho(b'')$, we find that $u'$ is the product of two $U_2$-elements of $\Ortho(b'')$, and hence $u$ is the product of two $U_2$-elements of
$\Ortho(b)$.
\end{proof}

\begin{cor}\label{cor:twistedpairs}
Let $(b,u)$ be a unipotent $1$-isopair with only Jordan cells of sizes $2k+1$ and
$2k-1$. Assume that $(b,u)_{t-1,2k+1} \simeq (b,u)_{t-1,2k-1}$. Then
$(b,u)$ is $U_2$-splittable.
\end{cor}

\begin{proof}
Let us take an orthogonal basis $(e_1,\dots,e_n)$ for the Wall bilinear form $(b,u)_{t-1,2k+1}$. Denote by
$\alpha_1,\dots,\alpha_n$ the values
 of the associated quadratic form at the vectors $e_1,\dots,e_n$.
By assumption, we can find an orthogonal basis $(f_1,\dots,f_n)$ for the form $(b,u)_{t-1,2k-1}$
such that the values of the associated quadratic form at $f_1,\dots,f_n$ are $\alpha_1,\dots,\alpha_n$.
By the classification of $1$-isopairs, we can find, for all $i \in \lcro 1,n\rcro$,
a unipotent $1$-isopair $(b_i,u_i)$ with exactly two Jordan cells, one of size $2k+1$ and one of size $2k-1$,
and both invariants $(b_i,u_i)_{t-1,2k+1}$ and $(b_i,u_i)_{t-1,2k-1}$ represent $\alpha_i$.
Then $(b,u) \simeq (b_1,u_1) \bot \cdots \bot (b_n,u_n)$.
By Corollary \ref{cor:twistedcor}, each pair $(b_i,u_i)$ is $U_2$-splittable, and we conclude
that so is $(b,u)$.
\end{proof}

\subsection{Reconstructing $U_2$-splittable pairs}\label{section:reconstruction}

We now have all the tools to prove the implication (i) $\Rightarrow$ (ii) in Theorem \ref{theo:orthogonal}.

We start with a simple case:

\begin{prop}\label{prop:Jordanevensize}
Let $(b,u)$ be a unipotent $1$-isopair whose Wall invariants are all hyperbolic.
Then $(b,u)$ is $U_2$-splittable.
\end{prop}

\begin{proof}
By the classification of $1$-isopairs, the Jordan numbers of the form $n_{t-1,2k}(u)$ are all even,
and by the assumption of hyperbolicity, all the Jordan numbers of the form $n_{t-1,2k+1}(u)$ are even.
Hence we can choose a unipotent automorphism $v$ of a vector space $V$ such that $n_{t-1,k}(v)=\frac{1}{2}n_{t-1,k}(u)$ for all $k \geq 1$. Then $h(v)$ is unipotent with the same Jordan numbers as $u$. We know from Proposition \ref{prop:Wallinvariantshyperbolic} that all the Wall invariants of $H_1(v)$ are hyperbolic. Since hyperbolic forms are equivalent if and only if they have the same rank, we deduce from the classification of $1$-isopairs that $(b,u) \simeq H_1(v)$.

Besides, $v$ is unipotent and hence it is the product of two $U_2$-elements of $\GL(V)$ (see Theorem \ref{theo:Botha}). Hence $h(v)$ is the product of two $U_2$-elements of $\Ortho(H_V^1)$,
by Proposition \ref{prop:hyperbolicexpansionU2split}. We conclude that $(b,u)$ is $U_2$-splittable.
\end{proof}

\begin{prop}\label{prop:Jordanoddsize}
Let $(b,u)$ be a $1$-isopair, where $u$ is nilpotent with only Jordan cells of odd size,
and $(-1)^k (b,u)_{t-1,2k+1}$ Witt-simplifies $\underset{i > k}{\bot}(-1)^i (b,u)_{t-1,2i+1}$ for every integer $k \geq 0$.
Then $(b,u)$ is $U_2$-splittable.
\end{prop}

\begin{proof}
We shall say that a $1$-isopair $(b',u')$ satisfies condition $(\calC)$ whenever
$(-1)^k (b,u)_{t-1,2k+1}$ Witt-simplifies $\underset{i > k}{\bot} (-1)^i\,(b,u)_{t-1,2i+1}$ for every integer $k \geq 0$.

We prove the result by induction on the dimension of the underlying space. The result is obvious for $0$-dimensional spaces, and more generally
when $u=\id$ (in that case it suffices to take the factors equal to $\id$).

So now we assume that $u \neq \id$. Because all the Jordan cells of $u$ have odd size, the nilindex of $u-\id$ reads $2\ell+1$ for some $\ell \geq 1$.
Write $B_i:=(b,u)_{t-1,2i+1}$ for all $i \geq 0$  and note that $B_{\ell}$ is non-zero, whereas $B_k=0$ for all $k>\ell$.

\noindent \textbf{Case 1: $B_\ell$ is isotropic.} \\
Then we can split $B_\ell=H \bot B'$ where $H$ is a hyperbolic form of rank $2$ and $B'$ is a non-degenerate symmetric bilinear form. Using the classification of $1$-isopairs, this helps us split $(b,u)=(b_1,u_1) \bot (b_2,u_2)$
in which:
\begin{itemize}
\item $(b_1,u_1)$ is a $1$-isopair with exactly two Jordan cells, both of size $2\ell+1$, and
$(b_1,u_1)_{t-1,2\ell+1}$ is hyperbolic;
\item $(b_2,u_2)$ is a $1$-isopair with only Jordan cells of odd size,
$(b_2,u_2)_{t-1,2\ell+1} \simeq B'$ and $(b_2,u_2)_{t-1,2k+1} \simeq  (b,u)_{t-1,2k+1}$ for all $k \neq \ell$.
\end{itemize}

By Proposition \ref{prop:Jordanevensize}, $(b_1,u_1)$ is $U_2$-splittable.

Next, we show that $(b_2,u_2)$ satisfies condition $(\calC)$: for all $k\neq \ell$,
the forms $\underset{i>k}{\bot} (-1)^i (b_2,u_2)_{t-1,2i+1}$ and $\underset{i>k}{\bot} (-1)^i B_i$
have the same non-isotropic parts up to equivalence, whereas $(-1)^k (b_2,u_2)_{t-1,2k+1} \simeq (-1)^k B_k$;
moreover $\underset{i>\ell}{\bot} (-1)^i (b_2,u_2)_{t-1,2i+1}$ equals zero, and hence every non-isotropic part of it is equivalent to a subform of
$(-1)^\ell (b_2,u_2)_{t-1,2\ell+1}$. Using the fact that $(b,u)$  satisfies condition $(\calC)$, we deduce that so does
$(b_2,u_2)$.

Hence, by induction $(b_2,u_2)$ is $U_2$-splittable, and we conclude that so is $(b,u)$.

\vskip 3mm
\noindent \textbf{Case 2: $B_\ell$ is nonisotropic.} \\
By condition $(\calC)$, we can split $B_{\ell-1} \simeq B_\ell \bot \varphi$ for some non-degenerate symmetric bilinear form $\varphi$.
This helps us split $(b,u)=(b_1,u_1) \bot (b_2,u_2)$
in which:
\begin{itemize}
\item $(b_1,u_1)$ is a $1$-isopair with only Jordan cells of size $2\ell-1$ and $2\ell+1$
$(b_1,u_1)_{t-1,2\ell+1} \simeq B_\ell$ and $(b_1,u_1)_{t-1,2\ell-1} \simeq B_\ell$;
\item $(b_2,u_2)$ is a $1$-isopair with only Jordan cells of odd size at most $2\ell-1$,
$(b_2,u_2)_{t-1,2\ell-1} \simeq \varphi$ and
$(b_2,u_2)_{t-1,2k+1} \simeq B_k$ for every $k \in \lcro 0,\ell-2\rcro$.
\end{itemize}
By Corollary \ref{cor:twistedcor}, $(b_1,u_1)$ is $U_2$-splittable.

To check that $(b_2,u_2)$ satisfies condition $(\calC)$, we note first that, for all
$k \leq \ell-2$, we have
$\underset{i >k}{\bot}(-1)^i B_i \simeq \bigl((-(-1)^\ell B_\ell) \bot ((-1)^\ell B_\ell)\bigr)
\bot \underset{i >k}{\bot}(-1)^i (b_2,u_2)_{t-1,2i+1}$, and since the form
$(-(-1)^\ell B_\ell) \bot ((-1)^\ell B_\ell)$ is hyperbolic this shows that the nonisotropic parts of
$\underset{i >k}{\bot}(-1)^i B_i$ are equivalent to those of $\underset{i >k}{\bot}(-1)^i (b_2,u_2)_{t-1,2i+1}$.
This is enough to check condition $(\calC)$ for $(b_2,u_2)$ at each integer $k \leq \ell-2$.
Besides, for $k\geq \ell-1$ we have $\underset{i >k}{\bot}(-1)^i (b_2,u_2)_{t-1,2i+1}=\{0\}$ and hence condition
$(\calC)$ is trivially satisfied at $k$ for $(b_2,u_2)$. We conclude that $(b_2,u_2)$ satisfies condition $(\calC)$.
Hence, by induction $(b_2,u_2)$ is $U_2$-splittable, and we conclude that so is $(b,u)$.

Hence, our inductive proof is complete.
\end{proof}

Now, we complete the proof of (i) $\Rightarrow$ (ii) in Theorem \ref{theo:orthogonalunipotent}.
So, let $(b,u)$ be a unipotent $1$-isopair in which, for all
$k \geq 1$, the symmetric bilinear form $(-1)^k\,(b,u)_{t-1,2k+1}$ Witt-simplifies $\underset{i>k}{\bot} (-1)^i\, (b,u)_{t-1,2i+1}$.
Using the classification of indecomposable $1$-isopairs, we can split
$(b,u) \simeq (b_1,u_1) \bot (b_2,u_2)$ in which
$(b_1,u_1)$ (respectively, $(b_2,u_2)$) is a unipotent $1$-isopair with all Jordan cells of odd (respectively, of even) size.

By Proposition \ref{prop:Jordanevensize}, the pair $(b_2,u_2)$ is $U_2$-splittable.
Next, we see that $(b_1,u_1)_{t-1,2k+1} \simeq (b,u)_{t-1,2k+1}$ for every integer $k \geq 0$, and it follows that
$(b_1,u_1)$ satisfies the assumptions of Proposition \ref{prop:Jordanoddsize}.
Hence, $(b_1,u_1)$ is $U_2$-splittable, and we conclude that so is $(b,u)$.

Thus, it is now proved that in Theorem \ref{theo:orthogonalunipotent} condition (i) implies condition (ii).
The proof of the converse implication is spread over the next (and last) three sections.

\subsection{Three reduction techniques to lower the nilindex}\label{section:induced}

Here, we shall examine three techniques to reduce the nilindex in the analysis of a unipotent $1$-isopair $(b,u)$.

\subsubsection{Descent}\label{section:descent}

Let $(b,u)$ be a unipotent $1$-isopair in which the nilindex of $u-\id$ equals $\nu \geq 3$.
We denote the underlying vector space by $V$.

Here, we compute the invariants of the induced pair $(b,u)^{\im (u-\id)}$, which we will locally denote by $(\overline{b},\overline{u})$.
Noting that $(\im (u-\id))^{\bot_b}=\Ker (u-\id)$, we see that the underlying vector space of the pair $(b,u)^{\im (u-\id)}$ is
$W:=\im (u-\id)/(\Ker (u-\id) \cap \im (u-\id))$.

Obviously $\overline{u}^{\nu-1}=0$. Let us look at the Wall invariants $(b,u)_{t-1,r}$.
To this end, we start by examining the case where all the Jordan cells of $u$ have the same size.

So, assume that, for some integer $r \geq 1$, all the Jordan cells of $u$ have size $r$.
If $r=1$ then $\im (u-\id)=\{0\}$ and $(\overline{b},\overline{u})$ is the zero pair.
If $r=2$ then $\im (u-\id)=\Ker (u-\id)$ and again $(\overline{b},\overline{u})$ is the zero pair.

Assume now that $r>2$. If $r$ is even, all the Wall invariants of $(\overline{b},\overline{u})$ equal zero.
Assume now that $r$ is odd.
Then the form $(\overline{b},\overline{u})_{t-1,r}$ is the sole non-zero Wall invariant of $(\overline{b},\overline{u})$.
We have $\Ker (u-\id)^k=\im (u-\id)^{r-k}$ for all $k \in \lcro 0,r\rcro$, and $\Ker (u-\id)^k=V$ for all $k \geq r$.
Hence, the source space of $\overline{u}-\id$ is $\Ker (u-\id)^{r-1}/\Ker (u-\id)$.
Then, $\Ker (\overline{u}-\id)^i=\Ker (u-\id)^{i+1}/\Ker (u-\id)$ for all $i \geq 0$.
It follows that all the Jordan cells of $\overline{u}-\id$ have size $r-2$ (and there are as many of them as there are Jordan cells of $u$).
Hence, $(\overline{b},\overline{u})_{t-1,r-2}$ is the sole non-zero Wall invariant of $(\overline{b},\overline{u})$.
It is defined as the bilinear form induced by $(x,y) \mapsto \overline{b}(x,(\overline{u}+\overline{u}^{-1}-2\id)^{(r-3)/2}(y))$ on the quotient space of $\Ker (u-\id)^{r-1}/\Ker (u-\id)$ with $\Ker (u-\id)^{r-2}/\Ker (u-\id)$, and it is therefore naturally isometric to the bilinear form
induced by $(x,y) \mapsto b(x,(\overline{u}+\overline{u}^{-1}-2\id)^{(r-3)/2}(y))$ on the quotient space $\Ker (u-\id)^{r-1}/\Ker (u-\id)^{r-2}$.

Set $v:=u+u^{-1}-2\id$.
The linear map $\alpha:=u- \id$ induces an isomorphism from $V/\Ker (u-\id)^{r-1}$ to $\Ker (u-\id)^{r-1}/\Ker (u-\id)^{r-2}$.
Noting that
$$\forall (x,y)\in V^2, \;
b\bigl(\alpha(x),v^{(r-3)/2}(\alpha(y))\bigr)=b(x,-v^{(r-1)/2}(y))=-b(x,v^{(r-1)/2}(y)),$$
we find that this isomorphism defines an isometry from $V/\Ker (u-\id)^{r-1}$ equipped with $(\overline{x},\overline{y})\mapsto b(x,v^{(r-1)/2}(y))$ to
$\Ker (u-\id)^{r-1}/\Ker (u-\id)^{r-2}$ equipped with $(\overline{x},\overline{y})\mapsto -\overline{b}(x,v^{(r-3)/2}(y))$.
Therefore, $(\overline{b},\overline{u})_{t-1,r-2} \simeq -(b,u)_{t-1,r}$.

Let us come back to the general case where one only assumes that $u$ is unipotent.
The pair $(b,u)$ is known to split into $(b,u) \simeq \underset{i \geq 1}{\bot} (b_i,u_i)$
where each $(b_i,u_i)$ is a $1$-isopair, $n_{t-1,k}(u_i)=0$ for all $k \neq i$, $(b_i,u_i)_{t-1,i} \simeq (b,u)_{t-1,i}$ for all odd $i$,
and $n_{t-1,i}(u)=n_{t-1,i}(u_i)$ for all even $i$.
Clearly, the previous reduction is compatible with isomorphisms and orthogonal direct sums. Compiling the information we have obtained on
the above special cases, we obtain the following conclusion:

\begin{prop}\label{proposition:descenteunipotent1}
Let $(b,u)$ be a unipotent $1$-isopair.
Set $(\overline{b},\overline{u})=(b,u)^{\im (u-\id)}$.
Then:
\begin{enumerate}[(i)]
\item For all odd $r \geq 1$, $(\overline{b},\overline{u})_{t-1,r} \simeq - (b,u)_{t-1,r+2}$.
\item The nilindex of $\overline{u}-\id$ equals $\max(1,\nu-2)$, where $\nu$ stands for the nilindex of $u-\id$.
\end{enumerate}
\end{prop}

\subsubsection{Twisted descent}\label{section:twisteddescent}

Here, we use a slightly different process.
Again, we let $(b,u)$ be a unipotent $1$-isopair.
This time around, we consider the induced pair $(\overline{b},\overline{u}):=(b,u)^{\Ker (u-\id)+\im (u-\id)}$.

Noting that $(\Ker (u-\id)+\im (u-\id))^{\bot_b}=\im (u-\id) \cap \Ker (u-\id)$ is included in $\Ker (u-\id)+\im (u-\id)$,
we see that the underlying vector space of $(\overline{b},\overline{u})$ is the quotient space
$(\Ker (u-\id) +\im (u-\id))/(\Ker (u-\id) \cap \im (u-\id))$.
From there, the results are very similar to the one of the previous section, with the notable exception
of the case where $u$ has only Jordan cells of size $1$ (i.e.\ $u=\id$).
In that case indeed, we have $(\overline{b},\overline{u})=(b,u)$, leading to
$(\overline{b},\overline{u})_{t,1}=(b,u)_{t-1,1}$ (instead of $(\overline{b},\overline{u})_{t-1,1}=0$).

Using the same method as in Section \ref{section:descent}, we obtain:

\begin{prop}\label{proposition:descentenilpotente2}
Let $(b,u)$ be a unipotent $1$-isopair.
Set $(\overline{b},\overline{u})=(b,u)^{\Ker (u-\id)+\im (u-\id)}$.
Then:
\begin{enumerate}[(i)]
\item For every odd $r \geq 2$, $(\overline{b},\overline{u})_{t-1,r} \simeq -(b,u)_{t-1,r+2}$.
\item $(\overline{b},\overline{u})_{t-1,1} \simeq (b,u)_{t-1,1}\, \bot \,-(b,u)_{t-1,3}$.
\end{enumerate}
\end{prop}

\subsubsection{Folding}\label{section:symmetrization}

Again, we let $(b,u)$ be a unipotent $1$-isopair. Let $k \geq 1$.

Here, we consider the induced pair $(b,u)^{\Ker (u-\id)^k}$, which is defined on the quotient space
$$W:=\Ker (u-\id)^k/(\im (u-\id)^k\cap \Ker (u-\id)^k)$$
because $(\Ker (u-\id)^k)^{\bot_b}=\im (u-\id)^k$.

Again, we compute the Wall invariants of $(\overline{b},\overline{u})=(b,u)^{\Ker (u-\id)^k}$ as functions of those
of $(b,u)$ by considering the special case where $u$ has only Jordan cells of size $r$ for some $r \geq 1$.
If $k\geq r$ then $\im (u-\id)^k \cap \Ker (u-\id)^k=\{0\}$ and $\Ker (u-\id)^k=V$, therefore
$(\overline{b},\overline{u})=(b,u)$.

Assume now that $k<r$. Then $\im (u-\id)^k=\Ker (u-\id)^{r-k}$ and hence
$\im (u-\id)^k \cap \Ker (u-\id)^k=\Ker (u-\id)^{r-k} \cap \Ker (u-\id)^k=\Ker (u-\id)^{\min(r-k,k)}$.
\begin{itemize}
\item If $k \leq r-k$ then $W=\{0\}$ and $(\overline{b},\overline{u})$ vanishes.
\item Assume that $k>r-k$. Then $W=\Ker (u-\id)^k/\Ker (u-\id)^{r-k}$, and $\overline{u}$ has only Jordan cells of size $2k-r$.
The invariant $(\overline{b},\overline{u})_{t-1,2k-r}$ is defined only if $r$ is odd.
So assume that $r$ is odd. Then the $(\overline{b},\overline{u})_{t-1,2k-r}$ invariant
is the bilinear form induced by
$(\overline{x},\overline{y}) \mapsto b(x,v^{(2k-r-1)/2}(y))$ on $\Ker (u-\id)^k/\Ker (u-\id)^{k-1}$, where $v:=u+u^{-1}-2\id$.
Noting that $b((u-\id)^{r-k}(x),v^{(2k-r-1)/2}((u-\id)^{r-k}(y)))=(-1)^{r-k} b(x,v^{(r-1)/2}(y))$ for all $x,y$ in $V$, we obtain that
the linear map $x \mapsto (u-\id)^{r-k}(x)$ induces an isometry
from $(b,u)_{t-1,r}$ to $(-1)^{r-k}(\overline{b},\overline{u})_{t-1,2k-r}$.
Hence, $(\overline{b},\overline{u})_{t-1,2k-r} \simeq (-1)^{r-k} (b,u)_{t-1,r}$, and it is the sole possible non-vanishing Wall invariant of $(\overline{b},\overline{u})$.
\end{itemize}

As in Section \ref{cor:twistedpairs}, we piece the previous results together to obtain the general form of the quadratic Wall invariants of $(\overline{b},\overline{u})$:

\begin{prop}\label{prop:symmetrization}
Let $(b,u)$ be a unipotent $1$-isopair. Let $k \geq 1$, and
set $(\overline{b},\overline{u})=(b,u)^{\Ker (u-\id)^k}$.
Then:
\begin{enumerate}[(i)]
\item For all odd $r \in \lcro 1,k-1\rcro$, $(\overline{b},\overline{u})_{t-1,r} \simeq
(b,u)_{t-1,r} \bot (-1)^{k+r}(b,u)_{t-1,2k-r}$.
\item $(\overline{b},\overline{u})_{t-1,k} \simeq (b,u)_{t-1,k}$ if $k$ is odd.
\item For all odd $r >k$, $(\overline{b},\overline{u})_{t-1,r}=0$.
\item The nilindex of $\overline{u}-\id$ is less than or equal to $k$.
\end{enumerate}
\end{prop}

\subsection{The key lemma: statement and consequences}\label{section:induction}

Our key result is the implication (i) $\Rightarrow$ (ii) in Theorem \ref{theo:orthogonalunipotent}
in the special case where $(u-\id)^3=0$.

\begin{lemma}[Key lemma]\label{lemma:keyunipotent}
Let $(b,u)$ be a $U_2$-splittable $1$-isopair in which $(u-\id)^3=0$.
Then $(b,u)_{t-1,1}$ Witt-simplifies $-(b,u)_{t-1,3}$.
\end{lemma}

The proof of this lemma is the most technical part of the present manuscript, so we wait until the next section to give it.
Here, we shall show how this lemma, combined with the results of Section \ref{section:induced},
helps one recover the necessary condition featured in Theorem \ref{theo:orthogonalunipotent}.

\begin{prop}\label{prop:CNunipotentcomplete}
Let $(b,u)$ be a $U_2$-splittable unipotent $1$-isopair.
Then $(-1)^k (b,u)_{t-1,2k+1}$ Witt-simplifies $\underset{i >k}{\bot} (-1)^i (b,u)_{t-1,2i+1}$ for all $k \geq 0$.
\end{prop}

Here, we assume the validity of Lemma \ref{lemma:keyunipotent},
and we prove Proposition \ref{prop:CNunipotentcomplete} by induction on the nilindex of $u-\id$ by steps of two.
If $(u-\id)^3=0$ then the result follows directly from Lemma \ref{lemma:keyunipotent} because in that case
$(b,u)_{t-1,2i+1}=0$ for all $i>1$, and hence the conclusion is limited to the statement that
$(b,u)_{t-1,1}$ Witt-simplifies $-(b,u)_{t-1,3}$.
Assume now that, for some integer $k \geq 2$, one has $(u-\id)^{2k+1}=0$.

We choose $U_2$-elements $u_1$ and $u_2$ of $\Ortho(b)$ such that $u=u_1u_2$.
By Lemma \ref{lemma:commutation}, $\im (u-\id)$ and $\Ker (u-\id)$ are stable under
$u_1$ and $u_2$.

Let us consider the induced $1$-isopair $(\overline{b},\overline{u}):=(b,u)^{\im (u-\id)}$, whose underlying vector space is
$\im (u-\id)/(\Ker (u-\id)\cap \im (u-\id))$. Since $\im (u-\id)$ is stable under $u_1$ and $u_2$,
we obtain (see Remark \ref{remark:inducedpairs}) that $(\overline{b},\overline{u})$ is $U_2$-splittable.
Then, $(\overline{u}-\id)^{2k-1}=0$, and by induction we recover that
$(-1)^\ell (\overline{b},\overline{u})_{t-1,2\ell+1}$ Witt-simplifies $\underset{i >\ell}{\bot} (-1)^i (\overline{b},\overline{u})_{t-1,2i+1}$ for all $\ell \geq 0$. Using Proposition \ref{proposition:descenteunipotent1}, we deduce that $(-1)^{\ell+1} (b,u)_{t-1,2\ell+1}$ Witt-simplifies $\underset{i >\ell}{\bot} (-1)^{i+1} (b,u)_{t-1,2i+1}$ for all $\ell \geq 1$, and hence $(-1)^\ell (b,u)_{t-1,2\ell+1}$ Witt-simplifies $\underset{i >\ell}{\bot} (-1)^{i} (b,u)_{t-1,2i+1}$ for all $\ell \geq 1$.

It only remains to obtain the conclusion for $\ell=0$. To do this, we go back to
$(b,u)$ and we apply the folding technique (Section \ref{section:symmetrization}) twice by
successively considering the induced pairs $(b',u'):=(b,u)^{\Ker (u-\id)^{2k}}$
and $(b'',u''):=(b',u')^{\Ker (u'-\id)^{2k-1}}$.
By Corollary \ref{cor:stabilizationlemma}, we know that $\Ker(u-\id)^{2k}$ is stable under $u_1$ and $u_2$, and
hence $(b',u')$ is $U_2$-splittable. We obtain likewise that $(b'',u'')$ is $U_2$-splittable.
Since $(u''-\id)^{2k-1}=0$, we find by induction that $(b'',u'')_{t-1,1}$ Witt-simplifies $\underset{i >0}{\bot} (-1)^{i} (b'',u'')_{t-1,2i+1}$.

Next, we note that $(b',u')_{t-1,2i+1}$ vanishes for all $i \geq k$, and it follows from Proposition \ref{prop:symmetrization}
that $(b',u')_{t-1,2i+1} \simeq (b'',u'')_{t-1,2i+1}$ for all $i \geq 0$.
Hence $(b',u')_{t-1,1}$ Witt-simplifies $\underset{i >0}{\bot} (-1)^{i} (b',u')_{t-1,2i+1}$.

Finally, noting that $(b,u)_{t-1,2l+1}=0$ for all $l > k$, we deduce from Proposition \ref{prop:symmetrization} that
$(b',u')_{t-1,2k-1} \simeq -(b,u)_{t-1,2k+1} \bot (b,u)_{t-1,2k-1}$, whereas
$(b',u')_{t-1,2i+1} \simeq (b,u)_{t-1,2i+1}$ for all $i \in \lcro 0,k-2\rcro$.
Hence $\underset{i >0}{\bot} (-1)^{i} (b,u)_{t-1,2i+1} \simeq \underset{i >0}{\bot} (-1)^{i} (b',u')_{t-1,2i+1}$,
whereas $(b,u)_{t-1,1} \simeq (b',u')_{t-1,1}$, and we conclude that
$(b,u)_{t-1,1}$ Witt-simplifies $\underset{i >0}{\bot} (-1)^i (b,u)_{t-1,2i+1}$.

This completes our inductive proof.

The proof of (ii) $\Rightarrow$ (i) in Theorem \ref{theo:orthogonalunipotent} is now entirely reduced to the proof of Lemma \ref{lemma:keyunipotent}, which
is detailed in the next and last section.

\subsection{Proof of the key lemma}\label{section:keylemmaproof}

We start with a purely algebraic result.

\begin{lemma}\label{lemma:skewcommutelemma}
Let $a,a'$ be square-zero elements of an $\F$-algebra $\mathcal{A}$ whose unity we denote by $\mathbf{1}$, set $u:=(\mathbf{1}+a)(\mathbf{1}+a')$ and assume that
$(u-\mathbf{1})^2=0$. Then, $a(u-u^{-1})=-(u-u^{-1})a$.
\end{lemma}

\begin{proof}
Setting $v:=u+u^{-1}-2.\mathbf{1}$, we find $v=aa'+a'a$, while $v=0$ because $(u-\mathbf{1})^2=0$.
It follows that $aa'a=-a'a^2=0$.
Next $au=a(\mathbf{1}+a')=a+aa'$ and
$$au^{-1}=a(\mathbf{1}-a')(\mathbf{1}-a)=a(\mathbf{1}-a-a'+a'a)=a-aa'+aa'a=a-aa',$$
and hence $a(u-u^{-1})=2aa'$. Likewise
$$ua=(\mathbf{1}+a)(\mathbf{1}+a')a=(\mathbf{1}+a+a'+aa')a=a+a'a+aa'a=a+a'a,$$
and $u^{-1} a=(\mathbf{1}-a')(\mathbf{1}-a)a=(\mathbf{1}-a')a=a-a'a$,
and hence $(u-u^{-1})a=2a'a$. The conclusion follows by using the fact that $a'a=-a'a$.
\end{proof}

Next, we recall a basic result that was already featured in \cite{dSPsquarezeroquadratic} (see step 2 in section 5.6 there):

\begin{lemma}\label{lemma:squarezero}
Let $b'$ be a non-degenerate symmetric (respectively, skewsymmetric) bilinear form
and $u'$ be a $b'$-skewselfadjoint (respectively, $b'$-selfadjoint) endomorphism such that $(u')^2=0$ and $u' \neq 0$.

Then there exists a $b'$-hyperbolic (respectively, a $b'$-symplectic) family $(x_1,x_2,y_1,y_2)$ such that $u'(x_2)=x_1$ and $u'(y_1)=-y_2$, leading to $u'(x_1)=0$ and $u'(y_2)=0$.
\end{lemma}

Remember that if $b'$ is skewsymmetric, a $b'$-symplectic family $(x_1,x_2,y_1,y_2)$ is one in which $b'(x_i,y_i)=-b'(y_i,x_i)=1$ for all $i \in \{1,2\}$, whereas all the other pairs of vectors
of $(x_1,x_2,y_1,y_2)$ are mapped by $b'$ to $0$.

\vskip 3mm
We are now ready to analyze the situation in Lemma \ref{lemma:keyunipotent}.
We will write $u=(\id+a)(\id+a')$, where $a,a'$ are $b$-skewselfadjoint square-zero endomorphisms.
We set $v:=u+u^{-1}-2\id=aa'+a'a$. By Lemma \ref{lemma:commutation}, $\id+a$ commutes with $v$, and hence $a$ commutes with $v$.
Note that $\Ker v=\Ker (u-\id)^2$.

Throughout, we set
$$B_1=(b,u)_{t-1,1} \quad \text{and} \quad B_3:=(b,u)_{t-1,3}.$$
First of all, as seen in Section \ref{section:results}, the conclusion of Lemma \ref{lemma:keyunipotent} means that
$\nu(B_3)+\nu(B_1 \bot (-B_3)) \geq \rk(B_3)$, and it is precisely this inequality we shall prove.
Note that it is obviously true if $B_3$ is hyperbolic.

\vskip 3mm
\noindent \textbf{Step 1: Reduction to the indecomposable case.}

Here, we will see that it suffices to consider the case where $(b,u)$ is \emph{indecomposable} as
a $U_2$-splittable $1$-isopair, meaning that
there do not exist non-trivial $1$-isopairs $(b_1,u_1)$ and $(b_2,u_2)$
such that $(b,u) \simeq (b_1,u_1) \bot (b_2,u_2)$ and each $(b_i,u_i)$ is $U_2$-splittable.

So, assume that such a decomposition exists and that the conclusion of Lemma \ref{lemma:keyunipotent} is valid for both $(b_1,u_1)$ and $(b_2,u_2)$.
Note that $B_1 \simeq (b_1,u_1)_{t-1,1} \bot (b_2,u_2)_{t-1,1}$ and
$B_3 \simeq (b_1,u_1)_{t-1,3} \bot (b_2,u_2)_{t-1,3}$.
Thus,
\begin{align*}
\nu(B_3)+\nu\bigl(B_1 \bot (- B_3)\bigr)
& \geq \nu\bigl((b_1,u_1)_{t-1,3}\bigr)+\nu\bigl((b_2,u_2)_{t-1,3}\bigr)\\
& \hskip 10mm +\nu\bigl((b_1,u_1)_{t-1,1}\bot - (b_1,u_1)_{t-1,3}\bigr) \\
& \hskip 10mm +\nu\bigl((b_2,u_2)_{t-1,1}\bot - (b_2,u_2)_{t-1,3}\bigr)
\end{align*}
and hence
$$\nu(B_3)+\nu\bigl(B_1\bot (- B_3)\bigr)
\geq \rk\bigl((b_1,u_1)_{t-1,3}\bigr)+\rk\bigl((b_2,u_2)_{t-1,3}\bigr)=\rk(B_3).$$
Therefore, by induction on the dimension of the underlying vector space, we see that it suffices to prove the conclusion
under the following additional conditions:
\begin{itemize}
\item[(I)] $(b,u)$ is indecomposable as a $U_2$-splittable $1$-isopair;
\item[(II)] The invariant $B_3$ is non-hyperbolic.
\end{itemize}
So, in the rest of the proof, we assume that conditions (I) and (II) are valid.

\vskip 3mm
\noindent \textbf{Step 2: $a$ maps $V$ into $\Ker v$.}

Since $a$ commutes with $v$, it stabilizes $\Ker v$
and hence it induces an endomorphism $\overline{a}$ of the quotient space $V/\Ker v$, which is the underlying vector space of the bilinear form $B_3$.
It turns out that $\overline{a}$ is $B_3$-skewselfadjoint. Indeed,
as $a$ commutes with $v$, we have
$$\forall (x,y)\in V^2, \quad b\bigl(x,v(a(y))\bigr)=b\bigl(x,a(v(y))\bigr)=- b(a(x),v(y)).$$
Now, we assume that $\overline{a} \neq 0$ and we seek to find a contradiction by applying Lemma \ref{lemma:squarezero}.

First of all, let $x \in V$ be an arbitrary vector. We shall see that
$$W_x:=\Vect\bigl(x,a(x),(u-\id)(x),(u-\id)(a(x)),v(x),v(a(x))\bigr)$$
is stable under both $a$ and $u$.
First of all, we note that
\begin{equation}\label{Wx2eexpresion}
W_x=\Vect\bigl(x,u(x),u^{-1}(x),a(x),u(a(x)),u^{-1}(a(x))\bigr).
\end{equation}
Since $(u-\id)^3=0$, we see that $u^2$ is a linear combination of $u,\id,u^{-1}$, and hence equality \eqref{Wx2eexpresion}
shows that $W_x$ is stable under $u$, and hence also under $u^{-1}$.
It remains to prove that $W_x$ is stable under $a$.
Using $a^2=0$ and the commutation of $a$ with $v$, it will suffice to prove that $(a(u-\id))(x)$ and $(a(u-\id)a)(x)$ belong to $W_x$.
To prove these facts, we start by collecting the following information:
\begin{enumerate}[(i)]
\item $(u-\id)(a(x))=a'a(x)+aa'a(x)$ belongs to $W_x$;
\item $u^{-1}(a(x))=(\id-a')(a(x))=a(x)-a'a(x)$ belongs to $W_x$;
\item $v(x)=aa'(x)+a'a(x)$ belongs to $W_x$.
\end{enumerate}
By combining point (ii) with the definition of $W_x$, we find that $a'a(x) \in W_x$.
Then $aa'(x) \in W_x$ by point (iii), and finally $aa'a(x) \in W_x$ by point (i).
As $a(u-\id)(x)=aa'(x)$ and $a(u-\id)a(x)=aa'a(x)$, we conclude that $W_x$ is stable under $a$.

We will now use the $W_x$ spaces to contradict assumption (I).

As $\overline{a}\neq 0$ and $\overline{a}$ is $B_3$-skewselfadjoint, Lemma \ref{lemma:squarezero} yields a $B_3$-hyperbolic family $(\overline{x_1},\overline{x_2},\overline{y_1},\overline{y_2})$ in $V/\Ker v$ such that $\overline{a}(\overline{x_2})=\overline{x_1}$ and $\overline{a}(\overline{y_1})=-\overline{y_2}$.

We choose arbitrary representatives $x_2$ and $y_1$ of $\overline{x_2}$ and $\overline{y_1}$ in $V$, and then we set $x_1:=a(x_2)$ and $y_2:=-a(y_1)$.
Now, we consider the space $U:=W_{x_2}+W_{y_1}$, which is stable under $u$ and $u_1=\id+a$, and hence under $u_1$ and $u_2=u_1^{-1} u$.
We will prove that $U$ is $b$-regular.
To see this, we consider the $12$-tuple
\begin{multline*}
(z_i)_{1 \leq i \leq 12} :=
\bigl(x_1,x_2,y_1,y_2,(u-\id)(x_1),(u-\id)(x_2),(u-\id)(y_1),(u-\id)(y_2),\\
v(x_1),v(x_2),v(y_1),v(y_2)\bigr).
\end{multline*}
Since $(u-\id)^3=0$, we have $\im (u-\id)^2 \subset \Ker (u-\id)$ and hence $\im (u-\id)^2 \bot_b \im (u-\id)$.
Noting that $\im v=\im (u-\id)^2$,
it follows that the matrix of $b$ in $(z_i)_{1 \leq i \leq 12}$ has the form
$$M=\begin{bmatrix}
? & ? & A \\
? & B & [0]_{4 \times 4} \\
A^T & [0]_{4 \times 4} & [0]_{4 \times 4}
\end{bmatrix}$$
where each block is a $4$-by-$4$ matrix.
Noting that $b((u-\id)(x),(u-\id)(y))=b(x,(u-\id)^\star (u-\id)(y))=-b(x,v(y))$ for all $x,y$ in $V$, we find that
$B=-A$.
Moreover, $A$ is precisely the matrix of $B_3$ in the hyperbolic family $(\overline{x_1},\overline{x_2},\overline{y_1},\overline{y_2})$, and
hence it is invertible. From this, we derive that $M$ is invertible.
We deduce that $(z_i)_{1 \leq i \leq 12}$
is a basis of $U$ and that $U$ is $b$-regular.
Hence $V=U \overset{\bot_b}{\oplus} U^{\bot_b}$, and $U^{\bot_b}$
is stable under both $u$ and $u_1$ because $u$ and $u_1$ are $b$-isometries.
It follows that both induced pairs $(b,u)^{U}$ and $(b,u)^{U^{\bot_b}}$ are $U_2$-splittable.
By assumption (I), we deduce that $U=V$ (because $U \neq \{0\}$).
Now, in that situation $B_3$ is hyperbolic,
thereby contradicting assumption (II).

\vskip 2mm
We conclude that $\overline{a}=0$, which means that $a$ maps $V$ into $\Ker v$.
It follows that $va=0$ and hence $av=0$ (see point (iii) in Lemma \ref{lemma:commutation}). Thus, $a$ vanishes everywhere on $\im v=\im(u-\id)^2$.

\vskip 3mm
\noindent \textbf{Step 3: $a$ maps $\Ker v$ into $\im (u-\id)+\Ker (u-\id)$.}

Remember from the start of Step 2 that $a$ stabilizes $\Ker v$.
Moreover, by Lemma \ref{lemma:commutation}, $a$ also stabilizes $\im (u-\id)$ and $\Ker (u-\id)$.
It follows that $a$ induces an endomorphism
$\overline{a}$ of the quotient space $V':=\Ker v /(\im (u-\id)+\Ker (u-\id))$.

Noting that $u-u^{-1}=u^{-1}(u+\id)(u-\id)$, we obtain that
$$(x,y) \mapsto b\bigl(x,(u-u^{-1})(y)\bigr)$$
induces a bilinear form on $V'$, denoted by $B_2$.
Because here $\im (u-\id) \subset \Ker (u-\id)^2$, we note that
$V'=\Ker (u-\id)^2/(\Ker (u-\id)+(\im(u-\id)\cap \Ker (u-\id)^2))$,
and just like in Section \ref{section:Wallinvariants} for the construction of the quadratic Wall invariants of a $-1$-isopair,
we find that $B_2$ is non-degenerate.
And finally $B_2$ is skewsymmetric because
$$\forall x \in V, \; b(x,(u-u^{-1})(x))=b(x,u(x))-b(x,u^{-1}(x))=b(x,u(x))-b(u(x),x)=0.$$
Hence, $B_2$ is a symplectic form.

Next, we prove that $\overline{a}$ is $B_2$-selfadjoint. To see this, we simply note that the endomorphism $u''$ of $\Ker v$ induced by $u$ satisfies
$u''+(u'')^{-1}=2\id$. Applying Lemma \ref{lemma:skewcommutelemma}, we deduce that the endomorphism $a''$ of $\Ker v$ induced by $a$
skewcommutes with $u''-(u'')^{-1}$, and hence
\begin{equation}\label{eq:skewcomm}
\forall x \in \Ker v, \quad a(u-u^{-1})(x)=-(u-u^{-1})(a(x)).
\end{equation}
It follows that
$$\forall (x,y)\in (\Ker v)^2, \; b\bigl(x,(u-u^{-1})(a(y))\bigr)=b\bigl(x,-(a(u-u^{-1}))(y)\bigr)
=b\bigl(a(x),(u-u^{-1})(y)\bigr).$$
Hence, $\overline{a}$ is $B_2$-selfadjoint, as claimed.

Assume now that $\overline{a} \neq 0$.
For $x \in \Ker v$, let us set
$$U_x:=\Vect\bigl(x,a(x),(u-u^{-1})(x),(u-u^{-1})(a(x))\bigr)\subset \Ker v.$$
Noting that $(u-u^{-1})^2=u^{-1}(u+\id)^2 v$ vanish everywhere on $\Ker v$,
we see that $U_x$ is stable under $u-u^{-1}$. As $u-u^{-1}$ coincides with $2(u-\id)$ on $\Ker v$,
we deduce that $U_x$ is stable under $u$. Using $a^2=0$ together with formula \eqref{eq:skewcomm},
we also find that $U_x$ is stable under $a$. We conclude that $U_x$ is stable under $\id+a$ and under $\id+a'=(\id+a)^{-1} u$.

This time around, we apply Lemma \ref{lemma:squarezero} (like in Step 2, but here $B_2$ is symplectic)
to the pair $(B_2,\overline{a})$ to obtain a family $(x_1,x_2,y_1,y_2)$ of vectors of $\Ker v$
whose family of classes modulo $\Ker (u-\id)+\im (u-\id)$ is a $B_2$-symplectic family and for which $a(x_2)=x_1$ and $a(y_1)=-y_2$.
We deduce that the family $(x_1,x_2,y_1,y_2,(u-u^{-1})(x_1),(u-u^{-1})(x_2),(u-u^{-1})(y_1),(u-u^{-1})(y_2))$ spans $U_{y_1}+U_{x_2}$
and that the matrix of $b$ in that family equals
$$M=\begin{bmatrix}
? & K' \\
(K')^T & [0]_{4 \times 4}
\end{bmatrix}$$ for $K':=\begin{bmatrix}
[0]_{2 \times 2} & I_2 \\
-I_2 & [0]_{2 \times 2}
\end{bmatrix}$, which is the matrix of $B_2$ in $(x_1,x_2,y_1,y_2)$.
Then $M$ is invertible.
Just like in Step 3, this leads to $U_{y_1}+U_{x_2}=V$ and further to $v=0$.
In particular $(u-\id)^2=0$ and hence $B_3=0$, contradicting assumption (II).
We conclude that $\overline{a}=0$, which means that
$a$ maps $\Ker v$ into $\Ker (u-\id)+\im (u-\id)$.

\vskip 3mm
\noindent \textbf{Step 4: Introducing an important linear map.}

Now, we look at the endomorphism $\overline{a}$ of $V/(\Ker (u-\id)+\im (u-\id))$ induced by $a$.
We know that its range is included in $\Ker v/(\Ker (u-\id)+\im (u-\id))$ and
that it vanishes everywhere on $\Ker v/(\Ker (u-\id)+\im (u-\id))$.
We start by choosing direct summands $D_1$ and $D_2$ as follows:
$$V/(\Ker (u-\id)+\im (u-\id))=\Ker \overline{a} \oplus D_1$$
and
$$\Ker \overline{a}=(\Ker v)/(\Ker (u-\id)+\im (u-\id)) \oplus D_{2.}$$
Then, we lift $D_1$ and $D_2$ to subspaces $V_1$ and $V_2$ of $V$ (so that the canonical projection from $V$ onto $V/(\Ker (u-\id)+\im (u-\id))$ maps bijectively $V_i$ onto $D_i$ for all $i \in \{1,2\}$). Hence
$$\Ker v \oplus V_1 \oplus V_2=V, \quad a(V_2) \subset \Ker (u-\id)+\im (u-\id),$$
$$(\Ker (u-\id)+\im (u-\id))\oplus a(V_1) \subset \Ker v \quad \text{and} \quad \dim a(V_1)=\dim V_1.$$
Next, we consider, like in Step 3, the symplectic form
$$B_2 : (\overline{x},\overline{y}) \mapsto b(x,(u-u^{-1})(y))$$ on
$\Ker v/(\Ker (u-\id)+\im (u-\id))$.
We consider the projection $\overline{a(V_1)}$ of $a(V_1)$ on $\Ker v/(\Ker (u-\id)+\im (u-\id))$,
and then a complementary subspace $D_3$ of the $B_2$-orthogonal $\overline{a(V_1)}^{\bot_{B_2}}$,
and finally we lift $D_3$ to a subspace $V_3$ of $\Ker v$ such that
$\dim V_3=\dim D_3=\dim(a(V_1))=\dim V_1$. Hence, the bilinear form $(x,y) \in a(V_1) \times V_3 \mapsto b(x,(u-u^{-1})(y))$
is non-degenerate on both sides.

Next, we consider the quotient space
$$Z:=(\Ker (u-\id)+\im (u-\id))/(\Ker (u-\id) \cap \im (u-\id)).$$
We have seen in Section \ref{section:twisteddescent} that $b$ induces a non-degenerate symmetric bilinear form $\overline{b}$ on $Z$, and that
$\overline{b}$ is equivalent to $B_1 \bot (-B_3)$.
We have $a(V_2) \subset \Ker (u-\id)+\im (u-\id)$, whereas $a(V_3) \subset \Ker (u-\id)+\im (u-\id)$ by Step 3.
Hence, we can consider the induced linear mapping
$$f : \begin{cases}
V_3\oplus V_2 & \longrightarrow Z \\
x & \longmapsto \overline{a(x)}.
\end{cases}$$
Since $a^2=0$, we have $\im a \subset \Ker a=(\im a)^{\bot_b}$, and hence
every vector in the range of $f$ is $\overline{b}$-isotropic.
It follows that
$$\nu\bigl(B_1 \bot (-B_3)\bigr) \geq \rk(f).$$
In order to conclude, it suffices to prove that $\rk(f) \geq \rk(B_3)-\nu(B_3)$.

\vskip 3mm
\noindent \textbf{Step 5: Proving that $\rk(f) \geq  \rk(B_3)-\nu(B_3)$.}

We start by checking that $\Ker f \cap V_3=\{0\}$. Let $x \in \Ker f \cap V_3$.
Then $a(x) \in \Ker (u-\id) \cap \im (u-\id)$ and in particular $(u-\id)(a(x))=0$, leading to
$(u-u^{-1})(a(x))=0$ because $u-u^{-1}=u^{-1}(u+\id)(u-\id)$.
Since $x \in \Ker v$, it follows from Lemma \ref{lemma:skewcommutelemma} that $a(u-u^{-1})(x)=0$.
Hence
$$\forall z \in V_1, \; b(a(z),(u-u^{-1})(x))=-b(z,a(u-u^{-1})(x))=0.$$
From the choice of $V_3$, we deduce that $x=0$.

Now, denote by $\pi$ the projection from $V_3 \oplus V_2$ onto $V_2$ along $V_3$.
Let $x \in \Ker f$, which we split into $x=x_3+x_2$ with $x_3 \in V_3$ and $x_2=\pi(x) \in V_2$.
First, we note that $a(x) \in \Ker (u-\id)$, leading to $a'(a(x))+(aa'a)(x)=0$.
Noting that $(aa'a)^\star=a^\star (a')^\star a^\star=(-1)^3 aa'a=-aa'a$,
we find $b(x,aa'a(x))=0$ and hence $b(x,a'(a(x)))=0$. It follows that $b(a'(x),a(x))=0$.
Finally
$$b(x,v(x))=b(x,(aa'+a'a)(x))=-b(a(x),a'(x))-b(a'(x),a(x))=-2b(a'(x),a(x))=0.$$
Since $V_3 \subset \Ker v$, it follows that the projection $\overline{\pi(\Ker f)}$ in $V/\Ker v$
of the subspace $\pi(\Ker f)$ is totally $B_3$-isotropic, yielding
$$\dim \overline{\pi(\Ker f)} \leq \nu(B_3).$$
Yet $\dim \overline{\pi(\Ker f)}=\dim \pi(\Ker f)$ because $V_2 \cap \Ker v=\{0\}$, whereas
$\dim \pi(\Ker f)=\dim \Ker f$ because $\Ker f \cap V_3=\{0\}$.
We conclude by the rank theorem that
\begin{align*}
\rk f & =\dim V_2+\dim V_3-\dim \Ker f \\
& \geq \dim V_2+\dim V_3-\nu(B_3) \\
& \geq \dim V_2+\dim V_1-\nu(B_3)=\dim(V/\Ker v)-\nu(B_3)=\rk (B_3)-\nu(B_3).
\end{align*}
Hence, the proof of the key lemma is finally complete, as well as the proof of Theorem \ref{theo:orthogonalunipotent}.
This completes our characterization of products of two unipotent elements of index $2$ in orthogonal groups.

\end{document}